\documentclass[sn-apa]{sn-jnl}
\usepackage{soul}
\usepackage{graphicx}
\usepackage{multirow}
\usepackage{amsmath,amssymb,amsfonts}
\usepackage{amsthm}
\usepackage{mathrsfs}
\usepackage[title]{appendix}
\usepackage{xcolor}
\usepackage{textcomp}
\usepackage{manyfoot}
\usepackage{booktabs}
\usepackage{algorithm}
\usepackage{algorithmicx}
\usepackage{lmodern}
\usepackage{algpseudocode}
\usepackage{listings}
\usepackage{subcaption}
\usepackage{pdflscape}
\usepackage{adjustbox}
\usepackage{rotating}
\usepackage{booktabs}
\usepackage{longtable}
\theoremstyle{plain}
\newtheorem{theorem}{Theorem}
\newtheorem{proposition}{Proposition}
\theoremstyle{definition}

\newcommand{\argmax}{\mathop{\mathrm{argmax}}\limits}
\theoremstyle{remark}
\newtheorem{definition}{Definition}

\begin{document}

\title[Article Title]{MTRBO: Multiple trust-region based Bayesian optimization}
\author[1]{\fnm{Sourav} \sur{Das}}\email{souravdas.maths@kgpian.iitkgp.ac.in}

\author[1]{\fnm{Debjani} \sur{Chakraborty}}\email{debjani@maths.iitkgp.ac.in}

\author[2]{\fnm{Pabitra} \sur{Mitra}}\email{pabitra@cse.iitkgp.ac.in}

\affil[1]{\orgdiv{Department of Mathematics}, \orgname{Indian Institute of Technology, Kharagpur}, \orgaddress{\city{Kharagpur}, \postcode{721302}, \state{West Bengal}, \country{India}}}


\affil[2]{\orgdiv{Department of Computer Science and Engineering}, \orgname{Indian Institute of Technology, Kharagpur}, \orgaddress{\city{Kharagpur}, \postcode{721302}, \state{West Bengal}, \country{India}}}
\abstract{Bayesian Optimization (BO) is a popular framework for optimizing black-box functions. Despite its effectiveness, BO is often inefficient for high-dimensional problems due to the exponential growth of the search space, heterogeneity of the objective function, and low sampling budget. To overcome these issues this work proposes a multiple trust region-based Bayesian optimization technique(MTRBO). A trust region is a localized region within which an optimization model is trusted to approximate the objective function accurately. Assuming a Gaussian process (GP) as a prior belief about the objective function and based on the posterior mean and variance functions, the method adaptively exploits near the promising current solution inside a trust region. Also explores the most uncertain region in the search space inside another trust region. The theoretical global convergence property of the proposed method is established. Then the work is benchmarked against other state-of-the-art trust-region-based Bayesian optimization algorithms, demonstrating superior performance on a variety of non-convex and high-dimensional test functions. The proposed method outperforms others in terms of solution quality within the sampling budget (the number of function evaluations). The proposed method is applied to portfolio optimization problem to verify its applicability in real-world scenarios.}


\maketitle

\section{Introduction}\label{Introduction}
In the field of engineering and scientific research, optimization often involves tackling computationally expensive (hard to evaluate, time-consuming, high evaluation cost, etc), black-box functions where traditional methods falter due to the lack of information about known mathematical properties like continuity, differentiability, convexity, etc. Bayesian Optimization (BO) has emerged as a powerful tool for these scenarios, generally using Gaussian processes to model and optimize complex objectives with limited evaluations. Bayesian optimization first originated from the work of \citep{kushner1964new}, where a Brownian motion stochastic process is assumed as prior for the objective function and then introduces the probability of improvement acquisition function for finding the location of the maximum point of an arbitrary multi-peak curve in the presence of noise. Another acquisition function, Expected improvement, is developed in \citep{movckus1975bayesian}. Though these are a few early works, Bayesian optimization got more attention after the work of \citep{jones1998efficient}, where the author proposed an efficient global optimization (EGO) algorithm for expensive black-box functions. Many variants of Bayesian optimization have been proposed over the years, \citep{du2022bayesian}, \citep{10093035}, and \citep{wang2024pre} being recent works. More details about recent works regarding Bayesian optimization can be found in \citep{wang2023recent}.
\par Despite its success across various fields like hyper-parameter tuning, robotics, material sciences, etc., BO faces challenges in high-dimensional search space settings, where its scalability and efficiency are tested. As \citep{eriksson2019scalable} pointed out, optimizing high-dimensional problems presents several challenges. Firstly, as the dimensionality increases, the search space expands exponentially, making it harder to locate global optima from increased local optima. Secondly, the function itself is often heterogeneous, which complicates the process of creating an effective global surrogate model. Lastly, the search space grows much faster than the sampling budget (total number of function evaluations possible) due to the curse of dimensionality, leading to regions with significant posterior uncertainty. This often causes common acquisition functions to excessively focus on exploration at the expense of exploiting potentially promising areas. Recent advancements aim to overcome these limitations with various different approaches. Numerous techniques leverage potential additive structures in the objective function. For example, \citep{kandasamy2015high}, \citep{gardner2017discovering}, \citep{wang2018batched}. These approaches often involve training a substantial number of Gaussian Processes (GPs), each representing different additive structures, which makes them less scalable to larger evaluation budgets. There are alternative methods \citep{wang2016bayesian}, \citep{nayebi2019framework} that depend on mapping the high-dimensional space to an unknown lower-dimensional subspace, allowing them to handle a large number of observations. Also, to handle a large number of observations, large-scale Bayesian optimization often involves selecting points in batches for parallel evaluation. Although various batch acquisition functions have been recently introduced \citep{chevalier2013fast}, \citep{shah2015parallel}, \citep{gonzalez2016batch}, these methods generally struggle to scale effectively with large batch sizes in practice. 
\par Though these approaches provide different directions to overcome the limitations of the traditional BO method, they are still not without limitations. Additive structure-based methods can often become computationally expensive and inefficient as the evaluation budget increases. Managing a large number of GPs is challenging and can lead to scalability problems with large evaluation budgets. Methods that map the high-dimensional space to a lower-dimensional subspace rely on specific assumptions about the structure of the objective function. These assumptions may not always hold in practice, leading to potential inaccuracies in the surrogate model. Although batch acquisition functions are designed to evaluate multiple points in parallel, they often face challenges when dealing with large batch sizes. Scaling these methods effectively for large batches remains problematic, as they may not perform well in practice due to increased computational requirements and difficulties in managing large numbers of parallel evaluations.
\par To overcome these limitations of BO and recent approaches to tackle the issues, this work proposes a multiple trust region-based Bayesian optimization (MTRBO) algorithm that considers two trust regions for exploration and exploitation separately in each iteration. For exploration, the region with the highest uncertainty after fitting a Gaussian process model based on the observed values to the expensive black-box objective function is considered. For exploitation, the trust region near the current best-observed value is considered, which is exploited for a few sub-iterations to find the maximizer of the posterior mean function, which itself converges to the actual objective function in the long run. So, at each iteration, the proposed method searches in a very small region compared to the search space without compromising on exploitation and exploration. The method is fine-tuned to not become sensitive to high exploration, which is a drawback of traditional BO in high-dimensional settings. At each iteration, the method first predicts two possible query points for the next iteration. One is from the exploration stage which is the maximizer of the acquisition function in the exploration trust region, and the other one is found exploiting near the current best-observed value, which maximizes the current posterior mean function over a few sub-iterations with different trust regions based on the ration of actual increment (on posterior mean) over predicted mean (on acquisition function). This is valid as in the long run, the posterior mean function converges to the actual objective function. Main contributions of this work are 
\begin{itemize}
    \item Proposing a global optimization method named multiple trust region-based Bayesian optimization (MTRBO).
    \item Establishing its theoretical global convergence property.
\end{itemize}
Recently, there have been a few developments on this approach \citep{regis2016trust}, \citep{eriksson2019scalable}, \citep{diouane2023trego}, \citep{li2023trust}. Comparisons of these approaches with the proposed approach are made in Section \ref{related work}.
\par The remainder of the work is organized as follows: Section \ref{Preliminaries} presents basic ideas about Gaussian process and traditional Bayesian optimization. Next, in Section \ref{MTRBO}, the proposed multiple trust region-based Bayesian optimization is discussed, and then Section \ref{Convergence analysis} provides the theoretical analysis of global convergence for the proposed method. And the results of the experiments, along with a comparison with existing trust region-based Bayesian optimization methods, are discussed in Section \ref{Experiments and results}. Finally, Section \ref{Conclusion} gives the overall conclusion of the proposed method.


\section{Related work}\label{related work}
TRIKE \citep{regis2016trust} employs a trust-region strategy in which each iteration is determined by maximizing an Expected Improvement (EI) function within a specified trust region. The size of this trust region is modified based on the ratio between the actual improvement and the predicted EI.
\par
TuRBO \citep{eriksson2019scalable} algorithm constructs a set of local models and strategically allocates samples among them using an implicit bandit method for global optimization.
\par
In TRLBO \citep{li2023trust}, two dynamically adjusting trust regions are employed to improve the algorithm's exploitation capabilities while maintaining its exploration potential. Specifically, one trust region helps minimize the number of samples in the Gaussian process, while the other limits the solution space for candidate points.
\par
TREGO \citep{diouane2023trego} alternates between regular efficient global optimization (EGO) steps and local steps within a trust region.
\par
All the previous work except TREGO, is mainly focused on the exploitation near the current best solution to reach a local optima. TRIKE uses a restart strategy by generating new initial observations if the EI falls below a threshold and exploits near the current best again. TuRBO does simultaneous Bayesian optimization runs with independent Gaussian process(GP) models, each within a different trust region. TRLBO uses a trust region to reduce the number of observations in the GP model, and another trust region to exploit near the current best observation. This method only focuses on the local optimization. Though TRIKE and TuRBO have some exploration potential, still exploration is not guaranteed. The argument behind the focus on exploitation only is that common acquisition functions mainly focus on exploration as the dimension of the search space increases. TREGO does consider the exploration but, this method does not specifically reduce the search space, by default at each iteration regular global search over the whole search space except for a few iterations when the global stage fails to improve the solution sufficiently. So, the fundamental problem, i.e, the Bayesian optimization not performing well in case of high-dimensional search space, is not properly solved. Some methods fail to explore the search space properly, some ignore the exploitation, leading to heavy exploration, and some techniques need additional properties of the objective function. The proposed MTRBO algorithm overcomes all the issues by considering two trust regions, one for exploring the region with the highest uncertainty, and the other one for exploiting near current best solution. The algorithm is fine-tuned such that it does not overly explore (tackling the issue pointed out in TuRBO) by exploiting near the current best even if the exploration stage provides a better query point in terms of maximizing the acquisition function. After the exploitation is completed, if still that point is better then only the algorithm shifts to that point for the next iteration. MTRBO balances between exploration and exploitation and also the search space is reduced only to the trust regions.


\section{Preliminaries}\label{Preliminaries}
Throughout the work, without loss of generality, all optimization problems in this paper are formulated as maximization problems. If a minimization objective is encountered, it is internally converted by negating the function value.
\subsection{Gaussian process}\label{GP}
A Gaussian process, $\mathcal{G}\mathcal{P}$, is a generalization of the multivariate Gaussian distributions to infinitely many variables. It is distributed over functions. Formally we can say that,
\begin{definition}
Gaussian Process is a collection of random variables, any finite number of which are Multi-variate Gaussian.
\end{definition}
Assume a function $f : \mathcal{X} \rightarrow{} \mathbb{R}$ follows Gaussian process, i.e, $f(x) \sim\mathcal{G}\mathcal{P}(\mu , k)$. Where, $\mu(x) : \mathcal{X} \rightarrow{} \mathbb{R} $, defined as $\mu(x) = \mathbb{E}[f(x)]$,  is the mean function and $k : \mathcal{X} \times \mathcal{X} \rightarrow{}  \mathbb{R}$, is covariance or kernel function. Here $\mu(x)$ is the average function value of all functions present in the distribution at the point $x$ and $k(x, x^{'})$ represents the dependence between the function values at different input points. A kernel is usually chosen based on the assumption that two points are more correlated as their distance decreases. \par Consider a finite collection of $n$ points, $x_{1},x_{2},\cdots, x_{n} \in\mathcal{X}$. Then their function values are $f_{i} = f(x_{i})$, for $i = 1,2,\cdots, n$. From the assumption $f_{1:n} = [f_{1}, f_{2}, \cdots , f_{n}]$ are Jointly Gaussian with the mean vector  $\mu_{1:n} = [\mu_{1}, \mu_{2}, \cdots , \mu_{n}]$ and covariance matrix $\mathbf{K}_{i,j} = k(x_{i}, x_{j})$. Now for a new point $x$, its function value $f(x)$ also follows Gaussian distribution with the following mean and variance function.\\
\begin{equation}\label{eq:posterior mean}
  \mu_{n}(x) = \mu(x) + \mathbf{k}(x, x_{1:n})^{T}\mathbf{K}^{-1}(f_{1:n} - \mu_{1:n})
\end{equation}
\begin{equation}\label{eq:posterior variance}
  \sigma^{2}_{n}(x) = k(x,x) - \mathbf{k}(x, x_{1:n})^{T}\mathbf{K}^{-1}\mathbf{k}(x, x_{1:n})
\end{equation}
Here, $\mathbf{k}(x, x_{1:n}) = [k(x, x_{1}), k(x, x_{2}), \cdots , k(x, x_{n})]$. This mean and variance are called posterior mean and variance and the distribution is called posterior distribution. The posterior mean $\mu_{n}(x)$ is a kernel-dependent weighted average between the prior $\mu(x)$ and an estimate $f_{1:n}$ derived from the data and posterior covariance $\sigma^{2}_{n}(x)$ is nothing but a term subtracted from prior covariance $k(x,x)$ corresponding to the variance removed by the previously seen data.
\par The kernel function must be positive definite, in the sense that for any finite collection of points, the kernel matrix formed by pairwise evaluation is positive definite. There are several kernel functions, but this work will use a Squared exponential kernel, also known as the Radial basis kernel, due to its smoothness properties and extensive usage in GP literature, and it is computationally efficient. If prior knowledge about the objective is known that, it is highly fluctuating, then the Matern kernel will be a better choice. Squared exponential kernel is defined as 
\begin{equation}\label{eq:3.3}
  k(x,x^{'}) = s^{2}\exp\left({- \frac{(x - x^{'})^2}{2l^{2}}}\right)
\end{equation}
where $s$ is the scale factor and $l$ is the length scale.
\subsection{Bayesian optimization}\label{BO}
Bayesian Optimization is a sequential model-based\footnote{Sequential model-based optimization cycles through the process of fitting models and using them to decide which options to examine.} method for carrying out global optimization of unknown, expensive-to-evaluate, black-box objectives.   
A probabilistic model, which captures our belief about the behavior of the unknown objective function, and an acquisition function, which determines where to sample next, are the two main components of Bayesian optimization. After initializing a prior belief about the objective function $f$, which is usually a Gaussian process, $\mathcal{G}\mathcal{P}(\mu(x), k(x, x^{'}))$ and collecting $n$ sample points $x_{1}, x_{2}, \cdots, x_{n}$ with the function values $f(x_{1}), f(x_{2}), \cdots , f(x_{n})$, posterior distribution is updated using Equation \ref{eq:posterior mean} and \ref{eq:posterior variance}, which is used to find the maximizer $x_{*}$ of the acquisition function $\alpha : \mathcal{X} \rightarrow \mathbb{R}$. Then, the posterior is updated with the modified observed values. See Algorithm~\ref{alg:1}.

\begin{algorithm}
\caption{Bayesian Optimization}
\label{alg:1}
\begin{algorithmic}[1]
\State Assume Objective function $f$ follows a prior distribution.
\State Observe value of $f$ at $n$ points.
\While {Condition is true} 
\State Update Posterior distribution on $f$.
\State From the current Posterior distribution find the maximizer of the acquisition 
\Statex \hspace{0.4cm} function.
\State Find the function value of $f$ at the maximizer.
\EndWhile \\
\Return Point giving largest objective function value
\end{algorithmic}
\end{algorithm}
Until now only the statistical model has been discussed, which is mainly Gaussian Process and represents belief about the unknown objective function. However, the procedure to generate the sequence of points in each iteration is not described. Random selection of query points could be possible but that will be a waste, instead selection strategies also known as the acquisition function that uses the posterior model to guide the selection search are used.
\subsection{Acquisition functions}\label{AF}
In Bayesian optimization, the acquisition function determines how the parameter space should be searched with the help of posterior distribution. 
Improvement-based acquisition functions favor the points that are likely to produce improvement upon the previously observed best objective function value. Let $f_{m}^{*}=max(f(x_{n}))$ for all $n\leq m$ be the best value after $m$-th iteration. So at the $(m+1) $-th iteration if the query point is $x_{m+1}$ and the objective value is $f(x_{m+1})$ then there will be an improvement upon $f_{m}^{*}$, if $f(x_{m+1})-f_{m}^{*}> 0$ . 
Another acquisition function is named upper confidence bound, which is very simple but effective if the hyper-parameters are tuned well enough. The followings are two improvement-based acquisition function, which considers the probability of the improvement and expectation of the improvement. Also one optimistic policy-based acquisition function Upper confidence bound.
\subsubsection{Probability of improvement(PI)}
At $(m+1)$-th iteration for an arbitrary point $x\in \mathcal{X}$ the improvement upon  $f_{m}^{*}$ is $f(x)-f_{m}^{*}$. Since $f$ follows the Gaussian process with posterior mean and variance as given in Equation \ref{eq:posterior mean} and \ref{eq:posterior variance}, given the observations $\mathcal{D}=\{(x_{i}, f(x_{i}))\}_{i=1}^{m}$, probability of improvement will be
\begin{equation}\label{eq:3.4}
\alpha(x) = \mathbb{P}[f(x)-f_{m}^{*}> 0] = 1- \mathbb{P}[f(x)\leq f_{m}^{*}] = 1- \Phi\left(\frac{f_{m}^{*}- \mu_{*}(x)}{\sigma_{*}(x)}\right)
\end{equation}
where $\Phi$ is the Standard normal cumulative distribution function. Recall from Section \ref{BO} that, the acquisition function is then maximized to find the maximizer as the next query point. So, at $(m+1)$-th iteration the query point will be
\begin{equation}\label{eq:3.5}
x_{m+1}= argmax_{x}\alpha(x) 
\end{equation}
Though this early strategy in the literature PI \citep{kushner1964new} performs well if the target is known, however in general PI exploits highly  with less exploration which can lead the search procedure to be stuck at a local optimum. To address this problem the following acquisition function considers the expectation of the improvement. 
\subsubsection{Expected improvement(EI)}\label{EI}
The expected improvement acquisition function is defined as the expectation of the improvement of the current functional value over the current best functional value to be positive
\begin{equation}\label{eq:3.6}
\alpha_n(x) = \mathbb{E}[f(x)-f_{n}^{*}> 0] =  (f_{n}^{*}-\mu_{n}(x)) \Phi\left(\frac{f_{n}^{*}- \mu_{n}(x)}{\sigma_{n}(x)}\right) + \sigma_{n}(x) \phi\left(\frac{f_{n}^{*}- \mu_{n}(x)}{\sigma_{n}(x)}\right)
\end{equation}    
where $\Phi$ and $\phi$ are the standard normal cumulative(CDF) and probability density(PDF) function respectively. Similarly as done in Equation \ref{eq:3.5}, one can find the maximizer $x_{n+1}$ in this case also. Intuitively, it can be thought of as a weighted sum of the improvement and uncertainty with weights being standard normal CDF and PDF. i.e, this acquisition function is balancing between the exploitation near the current best objective value and exploring the points where the uncertainty is high. Uncertainty is high means that, the region, where less data has been observed.

\subsubsection{Upper confidence bound(UCB)}\label{ucb} Upper Confidence Bound is a popular optimistic way to balance between exploration and exploitation by considering the weighted sum of posterior mean and variance, defined as
\begin{equation}\label{eq:3.7}
\alpha_n(x) = \mu_{n}(x) + \beta\sigma_{n}(x)
\end{equation}
where $\beta$ is an unknown parameter that represents how much preference is given to the exploration while searching for the next query point.

\subsection{Trust-region based optimization}\label{TR}
Trust-region methods are used to handle optimization problems, by restricting the search for an optimal solution within a region where the model is trusted to be an accurate representation of the objective function. The earliest works on trust-region methods can be traced back to \citep{levenberg1944method} and popularized after the work of \citep{marquardt1963algorithm}, giving name to the method as Levenberg-Marquardt method. A detailed review of the trust region methods can be found in \citep{yuan2015recent}. Unlike line search methods, which choose a direction and then decide how far to go in that direction, trust-region methods determine a region around the current point where the model is a good approximation and then optimize within this region. 
At the $k$-th iteration, a trust-region algorithm for the general optimization problem
\begin{equation*}
    \max_{x \in \mathcal{X}} f(x),
\end{equation*}
where $f(x)$ is the objective function to be maximized and $\mathcal{X} \subseteq \mathbb{R}^D$ is the feasible set, 
obtains the next query point by solving the following trust-region sub-problem:
\begin{equation}
    \max_{x \in \mathcal{X}_k} \alpha_k(x) \text{ subject to } \|x-x_k\| \leq r_k,
    \label{eq:TR1}
\end{equation}
where $\alpha_k(x)$ is an approximation of the objective function $f(x)$ near the current iteration point $x_k$, in the trust region $\mathcal{X}_k$, $\|.\|$ is a norm in $\mathbb{R}^D$, and $r_k > 0$ is the trust-region radius. The size of the trust region is critical and is adjusted dynamically. If the approximation is good (based on some criteria), the region might be expanded; otherwise, it is contracted. A general framework for the trust region optimization method is given in Algorithm (\ref{alg:general_framework_trust_region}).

\begin{algorithm}
\caption{Framework for a trust region optimization}
\label{alg:general_framework_trust_region}
\begin{algorithmic}[1]
\State \textbf{Initialization:} Given $x_1$, construct $\alpha_1(d)$, $\|.\|$, and $r_1 > 0$, set $k := 1$
\While {Convergence condition is not true} 
\State Solve the sub-problem defined in Equation (\ref{eq:TR1}) to find the trial step $s_k$
\State Decision about the acceptance of the trial step. 
\State Define $\alpha_{k+1}(x)$ and find trust region radius $r_{k+1}$
\State k=k+1
\EndWhile 
\end{algorithmic}
\end{algorithm}
Though trust-region methods are useful for solving large-scale nonlinear optimization problems where derivatives are available, extensions of trust-region methods are also used in derivative-free optimization, where the objective function’s derivatives are not available, and surrogate models are employed. In this work, we will focus on the latter kind.
\begin{algorithm}[htbp]
\caption{Multiple trust-region based Bayesian optimization}\label{alg:MTRBO}
\begin{algorithmic}[1]
\State \textbf{Inputs: } \begin{itemize}
    \item Expensive black-box objective function $f: \mathcal{X}(\subseteq\mathbb{R}^D) \rightarrow \mathbb{R}$, where $\mathcal{X}$ is compact.
    \item Initial set of points $S = \{x_1, x_2, \cdots, x_{n_0}\} \subseteq \mathcal{X}$ and $B$ contains the extreme points of the smallest hyper-rectangle containing $\mathcal{X}$.
    \item $N, N^{'}$ are the total number of iterations for global search and exploitation near the current best solution respectively.
    \item Initial exploitation trust-region radius $r_0$.
\end{itemize}
\State \textbf{Output: } The best point discovered by the algorithm.
\State Compute $f(S)= \{f(x_1), f(x_2), \cdots, f(x_{n_0})\}$ and $f_{best}= \text{max}(f(S))$. Set $n= n_0$
\While{$n \leq N$}
\State Update the Gaussian process based on the currently available data \State $\mathcal{D}_n = \{(x_i, f(x_i)): i=1,2,\cdots,n\}$.
\State Posterior mean $\mu_{n}(x) = \mu(x) + \mathbf{k}(x, x_{1:n})^{T}\mathbf{K}^{-1}(f_{1:n} - \mu_{1:n})$
\State Posterior variance $ \sigma^{2}_{n}(x) = k(x,x) - \mathbf{k}(x, x_{1:n})^{T}\mathbf{K}^{-1}\mathbf{k}(x, x_{1:n})$. (Section (\ref{GP}))
\State \textbf{Exploration: } 
\State $d =\mathop{\max}\limits_{x_i\in S \cup B} \{\mathop{\min}\limits_{x_j \in S \cup B, j\neq i} \{d(x_i, x_j)\}\}$. ($d$ is the Euclidean distance in $\mathbb{R}^D$)
\State $x_p = \mathop{\arg max}\limits_{x_i\in S \cup B} \{\mathop{\min}\limits_{x_j \in S \cup B, j\neq i} \{d(x_i, x_j)\}\}$ and $x_q = \mathop{\arg min}\limits_{x_j \in S \cup B, j\neq p} d(x_p, x_j)$
\State $r_{explore}= \frac{d}{2}$ and $x_{explore} = \frac{x_p +x_q}{2}$
\State $\mathcal{R}_{explore} = \{x \in \mathcal{X} | d(x, x_{explore}) < r_{explore}\}$
\State $x_{*} = \mathop{\argmax}\limits_{x \in \mathcal{R}_{explore}} \left(\alpha_{n}(x)\right)$ (Section (\ref{AF}) discusses about $\alpha_{n}$)
\State \textbf{Exploitation near current best:  }

\State Set $m=0$
\State $\mathcal{R}_{exploit} =  \{x \in \mathcal{X} | d(x, x_{best}) < r_m\}$, where $x_{best} =  \mathop{\argmax}\limits_{x \in S} \left(f(S)\right)$.
\State $x_m = x_{best}$ 
\While{$m < N^{'}$}
\State $x_{m+1} = \Pi_{\mathcal{R}^{t,m}_{exploit}}(x_m + \delta_m \nabla \mu_t(x_m))$. where $\Pi$ is projection operator(see Section \ref{MTRBO_explanation})
\State $r_{m+1} =
\begin{cases}
\gamma_{\text{inc}} \cdot r_m & \text{if } \mu_t(x_{m+1}) - \mu_t(x_m) > \eta_{\text{inc}} \cdot r_m, \\
\gamma_{\text{dec}} \cdot r_m & \text{otherwise},
\end{cases}$
\State $\mathcal{R}_{exploit} = $Trust-region centered at $x_{m+1}$ with radius $r_{m+1}$
\State $m=m+1$
\EndWhile
\State $x_{**} = x_{m}$
\State $x_{n+1} = \argmax(\alpha_{n}(x_{*}) , \alpha_{n}(x_{**}))$
\State Compute $f(x_{n+1})$
\State Modify the observed data, $\mathcal{D}_{n+1}= \mathcal{D}_{n} \cup \{(x_{n+1} , f(x_{n+1}))\}$ and $S = S \cup \{x_{n+1}\}$
\State $n=n+1$
\EndWhile \\
\Return $x_{opt} = \mathop{\argmax}\limits_{x \in S} \left(f(S)\right)$ and $f_{opt} = f(x_{opt})$
\end{algorithmic}
\end{algorithm}
\section{Proposed trust-region based Bayesian optimization technique}\label{MTRBO}
As discussed in Section (\ref{Introduction}), the traditional Bayesian optimization cannot perform well if the dimension of the search space is high, generally above $20$. Another thing is that the underlying acquisition function which is used to guide the search in Bayesian optimization becomes multi-modal itself, making the optimization much more difficult. So, in each iteration instead of searching the whole space,  two regions where the objective function is trusted to have the possible optimal solution is searched based on the acquisition function. This reduces the search space considerably, also inside that region multi-modality of the acquisition function reduces, which makes the search easier and effective. Two trust regions are considered at each iteration; one is for exploration which provides information in the highest uncertain area of the search space while the other one is for exploitation, which focuses on the neighborhood of the current best solution. 
\subsection{Multiple trust region based Bayesian optimization(MTRBO) algorithm}\label{MTRBO_explanation}
Suppose, there is an expensive (hard or costly to evaluate), black-box (mathematical properties like continuity, differentiability, convexity, etc are unknown) objective function $f: \mathcal{X}(\subseteq\mathbb{R}^D) \rightarrow \mathbb{R}$, which is to be maximized, where $\mathcal{X}$ is compact subset of $\mathbb{R}^D$. Prior belief about the objective function is taken as a Gaussian process with the mean function $\mu(x): \mathcal{X} \rightarrow \mathbb{R}$ and a kernel function $k(x,x^{'}): \mathcal{X} \times \mathcal{X} \rightarrow \mathbb{R}$. 
The objective function follows a Gaussian process means, that for every point $x \in \mathcal{X}$, its functional value $f(x)$ follows Gaussian distribution. Also for any finite collection of $n_0$ points $\{x_1, x_2, \cdots, x_{n_0}\}$, their functional values $\{f(x_1), f(x_2), \cdots, f(x_{n_0})\}$ follows multivariate Gaussian distribution. If these are the initially observed points, then the corresponding posterior mean, $\mu_{n_0}(x)$ and posterior variance, $\sigma_{n_0}(x)$ are given by the Equations (\ref{eq:posterior mean}) and (\ref{eq:posterior variance}).
\begin{itemize}
    \item Now, at any iteration $k$, two best points that maximize the acquisition function in $\mathcal{R}_{explore}$ and $\mathcal{R}_{exploit}$ respectively are found. Among them, the point giving a higher acquisition function value is predicted as the query point for the next iterations. 
    \item For the \textbf{exploration stage}, the radius $r_{explore}$ of the trust region $\mathcal{R}_{explore}$ is taken as $\frac{d}{2}$ centered at $x_{explore}=\frac{x_p+x_q}{2}$, where $d =\mathop{\max}\limits_{x_i\in S \cup B} \{\mathop{\min}\limits_{x_j \in S \cup B, j\neq i} \{d(x_i, x_j)\}\}$ and $x_p, x_q$ are the points such that $d=d(x_p,x_q)$ which can be found with simple computations like $x_p = \mathop{\arg max}\limits_{x_i\in S \cup B} \{\mathop{\min}\limits_{x_j \in S \cup B, j\neq i} \{d(x_i, x_j)\}\}$ and $x_q = \mathop{\arg min}\limits_{x_j \in S \cup B, j\neq p} d(x_p, x_j)$. $B$ contains all the extreme points of the smallest hyper-cube containing the compact search space $\mathcal{X}$. Only the points inside the search space are considered for optimization purposes. $x_{*} = \mathop{\argmax}\limits_{x \in \mathcal{R}_{explore}} \left(\alpha_{k}(x)\right)$ is chosen as the best observation from the exploration stage. \par In this stage, the most unexplored region (where the number of currently observed values is less) in the search space is considered, which can be seen from the definition. Also, one thing to note is that, as the iteration progresses intuitively one can see that the uncertainty in the whole space reduces so does the radius of the exploration region.
    \item In the \textbf{exploitation stage}, multiple sub-iterations are performed near the current best-observed value to ensure that, before shifting to another point predicted by the exploration stage, there is not enough improvement in the solution for the current exploitation trust region. 
    \begin{enumerate}
        \item Inside $k$-th iteration, each sub-iteration $m$ of the exploitation stage finds $x_{m+1}$ as gradient accent step of the posterior mean in the region $\mathcal{R}_{exploit}$, which is defined by $\mathcal{R}_{exploit} =  \{x \in \mathcal{X} | d(x, x_{m}) < r_m\}$, where for  $m=0, x_{m} =  \mathop{\argmax}\limits_{x \in S} \left(f(S)\right)$.
        \item For $m=0,1,2,\cdots,(N'-1)$, the gradient accent step is defined by $x_{m+1} = \Pi_{\mathcal{R}^{t,m}_{exploit}}(x_m + \delta_m \nabla \mu_t(x_m))$. where $\Pi$ is projection operator given by, $\Pi_{\mathcal{R}^{t,m}_{exploit}}(x)=
\begin{cases}
\alpha_i & \text{if } x_i \leq \alpha_i, \\
\beta_i & \text{if } x_i \geq \beta_i, \\
x_i & \text{otherwise.}
\end{cases}
$ for $\mathcal{X}=\{x \in \mathbb{R}^n| \alpha_i \leq x_i \leq \beta_i\}$.
        \item Based on the sufficient improvement, the trust region radius is decreased or increased, given by the rule 
    $$
    r_{m+1} =
    \begin{cases}
    \gamma_{\text{inc}} \cdot r_m & \text{if } \mu_t(x_{m+1}) - \mu_t(x_m) > \eta_{\text{inc}} \cdot r_m, \\
    \gamma_{\text{dec}} \cdot r_m & \text{otherwise},
    \end{cases}
    $$
    \end{enumerate}  
    \par Here, finding the actual objective function value to calculate the actual increment is not desirable because the objective is an expensive function, so the increment is taken over the posterior mean function which tends to the actual function as the number of observed values increases. 
    \item After all the sub-iterations are done in the exploitation stage, the $x_{m}$ is predicted as the best-observed point in this stage and named $x_{**}$.
    \item Between $x_{*}$ and $x_{**}$, the point giving a higher acquisition function value is predicted as the query for the $(k+1)$-th iteration. Observed data is then modified, and based on that posterior mean and posterior variance functions are updated.
\end{itemize}
The algorithm is provided systematically in Algorithm (\ref{alg:MTRBO}).
\section{Convergence analysis}\label{Convergence analysis}
In this section, an analysis of the convergence properties of the proposed algorithm is done. Begin with establishing the conditions under which the posterior mean of the Gaussian process model converges to the true objective function $f(x)$. Then the behavior of the exploration and exploitation stages are examined, demonstrating how the algorithm effectively balances exploration and exploitation to cover the entire search space and refine the search around promising regions. Finally, using these results it is proved that the algorithm achieves global convergence, ensuring that the global optimum of the objective function is eventually identified as the number of iterations increases.
\begin{proposition}[Uniform Convergence of Posterior Mean]\label{proposition:1}

Let $\mathcal{X} \subset \mathbb{R}^d$ be a compact domain, and let $f : \mathcal{X} \to \mathbb{R}$ belong to the Reproducing Kernel Hilbert Space (RKHS) $\mathcal{H}_k$ associated with a continuous, positive-definite kernel $k$, with RKHS norm bounded as $\|f\|_{\mathcal{H}_k} \leq M$. Let $\mu_t(x)$ denote the posterior mean of a Gaussian Process trained on $t$ noise-free observations $y_i = f(x_i)$. Assume that the points $\{x_i\}_{i=1}^t \subset \mathcal{X}$ become dense in $\mathcal{X}$ as $t \to \infty$.

Then,
\[
\lim_{t \to \infty} \mu_t(x) = f(x) \quad \text{uniformly over } x \in \mathcal{X}.
\]
\end{proposition}

\begin{proof}
The proof follows from RKHS interpolation theory.  
In the noise-free setting, the posterior mean $\mu_t(x)$ of the Gaussian Process with kernel $k$ trained on $\{x_i, f(x_i)\}_{i=1}^t$ coincides with the interpolant of $f$ from the RKHS $\mathcal{H}_k$ through the same points:
\[
\mu_t(x) = \sum_{i=1}^t \alpha_i k(x, x_i),
\]
where the coefficients $\boldsymbol{\alpha} = K_t^{-1} \mathbf{f}_t$ with $[K_t]_{ij} = k(x_i, x_j)$ and $\mathbf{f}_t = [f(x_1), \ldots, f(x_t)]^\top$.

Define the power function $P_t(x)$ as
\[
P_t(x)^2 := k(x, x) - k(x,x_{1:t})^\top K_t^{-1} k(x,x_{1:t}),
\]
where $k(x,x_{1:t}) := [k(x, x_1), \ldots, k(x, x_t)]^\top$.

The interpolation error is bounded in terms of the RKHS norm of $f$ and the power function:
\[
|f(x) - \mu_t(x)| \leq \|f\|_{\mathcal{H}_k} \cdot P_t(x).
\]
 
Note that, $\|f\|_{\mathcal{H}_k} \leq M$. Then,
\[
\sup_{x \in \mathcal{X}} |f(x) - \mu_t(x)| \leq M \cdot \sup_{x \in \mathcal{X}} P_t(x).
\]

Since $\mathcal{X}$ is compact, and $\{x_i\}_{i=1}^t$ becomes dense in $\mathcal{X}$ as $t \to \infty$
\[
\lim_{t \to \infty} \sup_{x \in \mathcal{X}} P_t(x) = 0.
\]

Hence,
\[
\lim_{t \to \infty} \sup_{x \in X} |f(x) - \mu_t(x)| = 0,
\]
i.e., $\mu_t(x) \to f(x)$ uniformly over $x \in \mathcal{X}$.
\end{proof}

This result ensures that the GP surrogate model accurately approximates the objective function as more data becomes available, provided the data covers the space densely. The proof follows classical kernel interpolation theory and can be found in \cite{ berlinet2004reproducing}.

\begin{proposition}[Local Convergence in the Exploitation Stage]\label{proposition:2}
At each iteration $t$, define the exploitation trust region $\mathcal{R}^{t,m}_{\text{exploit}} = \mathcal{B}(x_m,r_m)\cap\mathcal{X}$ where, $\mathcal{B}(x_m,r_m)$ is an open ball centered at the current best point $x_{\text{m}}$, with radius $r_m$ at sub-iteration $m$ and $\mathcal{X}$ is compact subset of $\mathbb{R}^n$. Within these trust regions, the algorithm performs $m$ sub-iterations of gradient ascent on the continuously differentiable posterior mean function $\mu_t(x)$, updating the trust region radius solely based on posterior mean improvement.
If the gradient ascent step is 
$$
x_{m+1} = \Pi_{\mathcal{R}^{t,m}_{exploit}}(x_m + \delta_m \nabla \mu_t(x_m)),
$$ 
where $\Pi$ is projection operator given by, $\Pi_{\mathcal{R}^{t,m}_{exploit}}(x)=
\begin{cases}
\alpha_i & \text{if } x_i \leq \alpha_i, \\
\beta_i & \text{if } x_i \geq \beta_i, \\
x_i & \text{otherwise.}
\end{cases}
$ for $\mathcal{X}=\{x \in \mathbb{R}^n| \alpha_i \leq x_i \leq \beta_i\}$.
The trust region radius update rule is defined as:
$$
r_{m+1} =
\begin{cases}
\gamma_{\text{inc}} \cdot r_m & \text{if } \mu_t(x_{m+1}) - \mu_t(x_{m}) > \eta_{\text{inc}} \cdot r_m, \\
\gamma_{\text{dec}} \cdot r_m & \text{otherwise},
\end{cases}
$$
where $\gamma_{\text{inc}} > 1$, $\gamma_{\text{dec}} \in (0,1)$ and $\eta_{\text{inc}} > 0$ is a small improvement threshold and are fixed. $\delta_m$ is the step length in gradient ascent chosen using the maximization rule, i.e, $\delta_m$ is chosen such that $\mu_t(x_m + \delta_m \nabla \mu_t(x_m))= \max_{\delta > 0}(\mu_t(x_m + \delta \nabla \mu_t(x_m)))$.

Then the sequence $\{x_m\}_{m=1}^{\infty} \subseteq \cup\mathcal{R}^{t,m}_{\text{exploit}}$ converges to a constrained stationary point (i.e, satisfies first-order optimality with respect to the feasible region. Note that, for closed and bounded search space $\mathcal{X}$, constrained stationary points can also lie in the boundary) $x^t_{**} \in \cup\mathcal{R}^{t,m}_{\text{exploit}}$.
\end{proposition}
\begin{proof}

Let $x_m$ denote the $m$-th point generated during the exploitation sub-iterations using a gradient ascent step on the posterior mean $\mu_t(x)$ within the trust region $\mathcal{R}^{t,m}_{\text{exploit}}$. By construction,
$$
x_{m+1} = \Pi_{\mathcal{R}^{t,m}_{exploit}}(x_m + \delta_m \nabla \mu_t(x_m)),
$$ 
where $\delta_m > 0$ is the step length. Then
$$
\mu_t(x_{m+1}) > \mu_t(x_m), \quad \text{unless } \nabla \mu_t(x_m) = 0 \text{ or } x_{m} = \Pi_{\mathcal{R}^{t,m}_{exploit}}(x_m + \delta_m \nabla \mu_t(x_m)).
$$
Hence, the sequence $\{\mu_t(x_m)\}_{m=1}^{\infty}$ is strictly increasing unless a constrained stationary point (either interior stationary point or boundary point) is reached.

Since the exploitation trust region $\mathcal{R}^{t,m}_{\text{exploit}}$ is always centered around $x^t_{\text{best}} \in \mathcal{X}$ and $\mathcal{X}$ is compact, all sub-iterates $x_m$ lie in a compact subset of $\mathbb{R}^d$. Thus, the sequence $\{x_m\}$ is bounded, and by the Bolzano-Weierstrass theorem, it has at least one accumulation point.

The radius $r_m$ evolves according to:
$$
r_{m+1} =
\begin{cases}
\gamma_{\text{inc}} \cdot r_m & \text{ if } \mu_t(x_{m+1}) - \mu_t(x_{m}) > \eta_{\text{inc}} \cdot r_m, \\
\gamma_{\text{dec}} \cdot r_m & \text{otherwise},
\end{cases}
$$
with $\gamma_{\text{inc}} > 1$, $\gamma_{\text{dec}} \in (0,1)$. When $\mu_t(x_m)$ increases significantly, the region expands, allowing broader search. If not, the region shrinks, refining the search. Because the mean function is smooth and $\mathcal{X}$ is compact, this adaptive adjustment ensures convergence to a local optimum.

Since $\mu_t(x)$ is differentiable and hence continuous and $\{\mu_t(x_m)\}$ is monotonically increasing and bounded above (as $\mu_t$ is continuous over compact $\mathcal{X}$) converging to $\mu_t(x_{**}^t)$, then the sequence $\{x_m\}$ converges to the point $x^t_{**}$(note that, $\mu_t$ is locally strictly increasing at $x_{**}^t$).
Therefore, $x_{**}^t$ is a constrained stationary point of the posterior mean function. Thus, under the stated trust region update rule, the sequence $\{x_m\}$ converges to a constrained stationary point $x^t_{**}$ of the posterior mean $\mu_t(x)$ within the union of exploitation regions.
\end{proof}
For more details about the convergence analysis of projected gradient, please refer to \citep{bertsekas1997nonlinear}

\begin{proposition}[Exploration]\label{proposition:3}
Let $\mathcal{X}\subset \mathbb{R}^D$ be a compact hyper-rectangle.  
Define the set of evaluated points after $t$ iterations as 
$S_t=\{x_1,x_2,\dots,x_t\}\subseteq \mathcal{X}$, and let $B$ denote the set of extreme points (vertices) of $\mathcal{X}$.  
Set $A_t = S_t \cup B$.

For each $x_i \in A_t$, define the max-min distance 
$d_t =\mathop{\max}\limits_{x_i\in A_t} \{\mathop{\min}\limits_{x_j \in A_t, j\neq i} \{d(x_i, x_j)\}\}$.
Then select, $ x_p^t = \mathop{\arg max}\limits_{x_i\in A_t} \{\mathop{\min}\limits_{x_j \in A_t, j\neq i} \{d(x_i, x_j)\}\} \text{ and } x_q^t= \mathop{\arg min}\limits_{x_j \in A_t, j\neq p} d(x_p^t, x_j).$
Define the exploration trust region at iteration $t$ as
$$\mathcal{R}_{\mathrm{explore}}^t 
= \Big\{ x \in \mathcal{X} \;\Big|\; d(x,x_{\mathrm{explore}}^t) < r_{\mathrm{explore}}^t \Big\},$$
where the center and radius are
$$
x_{\mathrm{explore}}^t \;=\; \tfrac{1}{2}(x_p^t + x_q^t), ~
r_{\mathrm{explore}}^t \;=\; \tfrac{d_t}{2}.
$$

As $t \to \infty$, the iterative construction ensures that the union of exploration trust regions 
$\bigcup_{t=1}^\infty \mathcal{R}_{\mathrm{explore}}^t$ becomes dense in $\mathcal{X}$, thereby reducing posterior uncertainty uniformly and preventing omission of any potential global optimum $x_* \in \mathcal{X}$.
\end{proposition}
\begin{proof}
    After each iterations, a new point is added to the observed set, so, from the construction of $d_m$, the sequence $\{d_m\}$ is monotonically decreasing. It will become eventually constant only if the sequence of selected points $\{x_m\}$ converges to the global optimum, otherwise there exists a strictly decreasing subsequence $\{{d_m}_{k}\}_{k=1}^{\infty}$ of $\{d_m\}$ with zero as greatest lower bound. Hence, the sequence $\{d_m\}$ converges to zero. Therefore $\forall x \in \mathcal{X}$, and $\forall \epsilon>0, \exists ~ T \in \mathbb{N}$ such that $d(x,y)<\epsilon$ for some $y \in A_T = S_T \cup B = \{x_1, x_2, \cdots, x_T\} \cup B \subseteq \bigcup_{t=1}^\infty \mathcal{R}_{\mathrm{explore}}^t$. Hence, $\bigcup_{t=1}^\infty \mathcal{R}_{\mathrm{explore}}^t$ is dense in $\mathcal{X}$.
\end{proof}

\begin{theorem}[Global Convergence of the Algorithm]\label{theorem:1}
    Let, $f(x) : \mathcal{X} \to \mathbb{R}$ is an expensive black-box objective function, which is assumed to follow a Gaussian process with posterior mean $\mu_t(x)$ and posterior variance $\sigma_t^2(x)$. Then as $t \to \infty$ and $m \to \infty$, the proposed algorithm converges to a global optimum $x^{*}$ of the objective function. 
\end{theorem}
\begin{proof}
    From propositions (\ref{proposition:2}) and  (\ref{proposition:3}), it follows that proposed algorithm reached to local optima of the posterior mean function as number of sub-iteration for exploitation stage $m \to \infty$ and the search space is extensively explored as number of iteration of the algorithm $t \to \infty$. So, global convergence of $\mu_t(x)$ is established, which itself converges to the actual objective function $f(x)$ as $t \to \infty$ (from proposition (\ref{proposition:1})). Hence, the proposed algorithm converges globally.
\end{proof}
Throughout this work, it is assumed that the objective is an expensive black-box function but deterministic, and there is no noise in the observed functional values. In case of noisy observations, the convergence can be proved using the concentration inequality as discussed in \citep{chowdhury2017kernelized}. The global convergence for categorical and mixed search spaces is shown in \citep{wan2021think}. \\
\noindent\textit{Remark.} Theorem~1 establishes only the theoretical soundness of the proposed method, ensuring non-divergence, and is not intended as validation of its superiority. The main evidence supporting effectiveness of the proposed method tackling the issues of standard Bayesian optimization in high dimensional setting is provided by the empirical results in the following section.




\begin{table}[htbp]
\caption{Summary of All Experiments Conducted}
\label{table:experiment_summary}
\begin{tabular}{|p{1.6cm}|p{3cm}|p{2.5cm}|p{6.5cm}|}
\hline
\textbf{Experiment Type} & \textbf{Description} & \textbf{Test Cases / Datasets} & \textbf{Evaluation Metrics / Notes} \\
\hline

\textbf{Main Comparison} & Comparison of proposed method against baselines over 30 trials using normalized performance. See Tables (\ref{table:comparison_2d},\ref{table:comparison_20d},\ref{table:comparison_50d},\ref{table:comparison_100d},\ref{table:comparison_500d}) & 15 test functions × 5 dimensions (2, 20, 50, 100, 500) & Normalized scores: $\frac{f - f_{\min}}{f_{\max} - f_{\min}} \times 100$, 
$f_{\min}$ and $f_{\max}$ are the minimum and maximum functional value achieved by all the methods over all the independent trial runs. EI acquisition function is used as acquisition function. Hyperparameter values: $N=100, N'=15, r_0=0.5, n_0 = 10, \gamma_{\text{inc}}=2, \gamma_{\text{dec}}= 1/2, \eta_{\text{inc}}= 0.01, \delta =0.1$. Length scale($l$) of GP is determined via MLE\\
\hline

\textbf{Ablation Study} & Evaluate contributions of components of the proposed method. See Table (\ref{table:ablation_study}) & 2 test functions × 2 dimensions (e.g., Ackley, Eggholder in 2D, 20D) & Best / Mean / Worst values, 30 independent trials per configuration. Hyperparameter values are same as main comparison \\
\hline

\textbf{Statistical Significance} & Wilcoxon Rank-Sum test between proposed and baseline methods. See Tables (\ref{table:comparison_2d},\ref{table:comparison_20d},\ref{table:comparison_50d},\ref{table:comparison_100d},\ref{table:comparison_500d}) & Same as main comparison. Results are observed for  30 independent trials for each function and dimension.  & Symbols $+, -$ and $\approx$ indicate compared method is significantly worse(larger objective values), significantly better(smaller values), and no significant difference relative to MTRBO, respectively, at the $5\%$ significance level $(p<0.05)$; hyperparameters match the main comparison.\\
\hline

\textbf{MuJoCo HalfCheetah} & Performance of optimization methods on real-world RL benchmark problem. See Figure (\ref{image:mujoco}) & OpenAI Gym’s HalfCheetah-v2 & Cumulative reward  \\
\hline

\textbf{Portfolio Optimization} & Application to real-world stock portfolio allocation. See Table (\ref{table:portfolio})& Historical stock data from Indian National Stock Exchange and the New York Stock Exchange& Objective: maximize weighted sum of return and risk. Evaluation metric: Sharpe ratio, weight $\lambda=0.5$ \\
\hline

\textbf{Sensitivity Analysis} & Impact of hyperparameters on performance. See Table \ref{table:sensitivity_analysis} & Ackley function (2D) & Objective values for different set of hyperparameter values. $\gamma_{\text{inc}}=2, \gamma_{\text{dec}}= 1/2, \eta_{\text{inc}}= 0.01, \delta =0.1$ are fixed for all the analysis. Length scale in kernel function is treated as hyperparameter. \\
\hline

\textbf{Contour + Selected points} & Visualization of selected points in the search space overlaid on contour plots. See Figure (\ref{image:contour_1},\ref{image:contour_2}) & Six 2D functions & Points selected during exploration vs exploitation for an arbitrary trial run and 3d visualization of the objective function \\
\hline

\textbf{Convergence Curves} & Track progress of each method per iteration to show optimization speed. See Figure (\ref{image:progress}) & Ackley function (2D) & Mean best value vs iteration, together with best and worst values over 30 trials. The hyperparameters are same as the main comparison  \\
\hline

\end{tabular}
\end{table}

\section{Experiments and results}\label{Experiments and results}
The primary objective of the experiments is to evaluate the efficacy and performance of the proposed MTRBO algorithm in solving high-dimensional problems. Specifically, the aim is to assess the algorithm's ability to locate optimal or near-optimal solutions efficiently across various benchmark problems, including non-convex and non-differentiable functions. The experimental results will also provide insights into the scalability and robustness of MTRBO when applied to complex, expensive black-box functions. The experiment is multi-fold, first, the performance of the method is analyzed using synthetic test functions of different dimensions along with sensitivity analysis to hyper-parameters used in the method, ablation study to get the insights about the effect of components like exploration and exploitation in methods performance, and then applying the method to real world problems like MuJoCo and portfolio optimization to verify its acceptability in real-world applications. Also statistical significance is tested for the proposed method against baseline methods. The complete details about all the conducted experiments are presented in Table \ref{table:experiment_summary}. All the experiments are done in a PC with an Intel(R) Core(TM) i5-6500 processor, and 16 GB RAM running on Windows 11. 
\subsection{Test functions and algorithms compared}
17 synthetic test functions, each with different dimensions(2, 20, 50, 100, 500), have been used from a Python library called "OptimizationTestFunctions" \citep{OptimizationTestFunctions} for testing. Expressions for the objective functions that are optimized are given in \ref{TF}. The problem MuJoCo HalfCheetah-v2 from OpenAI's Gym library is used for real-world application of the proposed method, together with application to portfolio optimization, where historical data about stock prices is collected from the "yfinance" Python library.
\subsection{Experiment on test functions}
 As Bayesian optimization is designed to work well for expensive functions, the low sample budget is desirable, i.e, with very few function evaluations, predicting the actual objective function as accurately as possible. In this work, the sample budget is set as 100 for all the comparisons, and later sensitivity analysis is done to check its effect on the proposed method. Therefore, after 100 function evaluations for each test function with different dimensions (2, 20, 50, 100, 500), the objective function value of the proposed method, multiple trust-region Bayesian optimization (MTRBO), together with efficient global algorithm (EGO), trust-region implementation in kriging-based optimization with expected improvement (TRIKE), trust-region Bayesian optimization (TURBO), trust-region framework for efficient global optimization (TREGO), trust region based local Bayesian optimization (TRLBO), are observed and normalized in the region of minimum and maximum value observed by each methods over all independent trials. For all the experiments, the hyperparameters for the baseline methods are: For EGP-TS, number of random features(D)=5, number of workers (K)=2, number of GP ensembles M=2, number of initial points($n_0$)=10, number of iterations(N)=100. For TRLBO, initial trust region radius($r_0$)=0.5, UCB (acquisition function used in TRLBO) coefficient($\beta=0.5$), number of initial points ($n_0$), number of iteration (N)=100 for all the other methods the hyperparameters required are same as the proposed method discussed in main comparison of Table \ref{table:experiment_summary}. Different random samples are taken for all methods in all the experiments conducted, and for testing statistical significance, the Wilcoxon rank-sum test is used. Another important thing to note is that, for the proposed method, the Expected Improvement(EI) acquisition function is used in all the experiments and is optimized using the L-BFGS-B method of the SciPy Python library. Though one can use other acquisition functions like probability of improvement (PI), upper confidence bound (UCB) or more recent family of acquisition function like LogEI(\citep{ament2023unexpected}).
\begin{table}[h]
\centering
\caption{Ablation study results on Ackley and Eggholder functions.}
\label{table:ablation_study}
\begin{tabular}{|c|c|c|c|c|c|}
\hline
Function & Dimension & Stage Used & Best Value & Worst Value & Mean Value \\
\hline
\multirow{3}{*}{Eggholder} 
 & \multirow{3}{*}{2} 
 & Only Exploitation & -971.81 & -747.21 & -862.92 \\
 &  & Only Exploration  & -1056.86 & -816.24 & -918.74 \\
 &  & Both Enabled      & -1056.81 & -894.58 & \textbf{-979.57} \\
\hline
\multirow{3}{*}{Ackley} 
 & \multirow{3}{*}{2} 
 & Only Exploitation & 0.51 & 6.82 & 4.18 \\
 &  & Only Exploration  & 1.52 & 4.40 & 3.51 \\
 &  & Both Enabled      & 0.93 & 4.62 & \textbf{2.21} \\
\hline
\multirow{3}{*}{Eggholder} 
 & \multirow{3}{*}{20} 
 & Only Exploitation & -4006.50 & -3006.72 & -3647.63 \\
 &  & Only Exploration  & -8101.87 & -4938.23 & -6545.93 \\
 &  & Both Enabled      & -8287.67 & -5957.71 & \textbf{-6933.54} \\
\hline
\multirow{3}{*}{Ackley} 
 & \multirow{3}{*}{20} 
 & Only Exploitation & 62.53 & 87.41 & 77.85 \\
 &  & Only Exploration  & 10.80 & 70.46 & 45.14 \\
 &  & Both Enabled      & -14.58 & 22.69 & \textbf{10.80} \\
\hline
\end{tabular}
\end{table}


\begin{table}[]
\caption{Predicted mean optimal value over 30 trials for each method after 100 function evaluations for $2$-dimensional search space.}
\label{table:comparison_2d}
\begin{tabular}{llrrrrrrr}
\toprule
Function & Metric & MTRBO & EGO & TREGO & TRIKE & TURBO & TRLBO & EGP-TS \\
\midrule
Abs & Best & 0.00 & 0.27 & 14.23 & 0.97 & 0.87 & 14.82 & 0.14 \\
 & Mean & \textbf{0.86} &   1.06($\approx$) & 43.99(+) & 45.24(+) & 2.43(+) & 41.50(+) & 5.33(+) \\
 & Worst & 1.72 & 1.76 & 64.02 & 100.00 & 5.02 & 85.91 & 7.89 \\

Ackley & Best & 0.00 & 0.18 & 7.27 & 11.99 & 0.05 & 17.45 & 2.23 \\
 & Mean & 1.93 & 4.32(+) & 40.54(+) & 23.09(+) & \textbf{1.04}(-) & 45.88(+) & 15.46(+) \\
 & Worst & 3.70 & 18.40 & 83.37 & 57.21 & 2.01 & 100.00 & 28.65 \\

AckleyTest & Best & 0.00 & 1.78 & 20.42 & 25.90 & 3.09 & 21.04 & 0.12 \\
 & Mean & \textbf{3.59} & 19.59(+) & 66.80(+) & 60.58(+) & 11.27(+) & 54.25(+) & 8.54($\approx$) \\
 & Worst & 12.24 & 32.53 & 100.00 & 91.95 & 18.78 & 91.26 & 35.41 \\

Eggholder & Best & 0.00 & 2.75 & 30.62 & 26.06 & 3.15 & 24.31 & 94.03 \\
 & Mean & \textbf{13.39 }& 23.10($\approx$) & 52.44(+) & 53.11(+) & 21.74(+) & 56.63(+) & 97.62(+) \\
 & Worst & 24.53 & 34.66 & 76.18 & 85.13 & 41.53 & 89.02 & 100.00 \\

Griewank & Best & 0.05 & 2.24 & 3.34 & 7.91 & 0.24 & 2.86 & 0.00 \\
 & Mean & 0.96 & 6.02(+) & 37.11(+) & 49.02(+) & 1.07(-) & 28.12(+) & \textbf{0.21}(-) \\
 & Worst & 3.40 & 13.80 & 93.68 & 100.00 & 2.53 & 91.31 & 0.52 \\
 
Penalty2 & Best & 0.00 & 0.11 & 13.07 & 2.04 & 0.10 & 2.24 & 0.01 \\
 & Mean & \textbf{0.05} & 0.53(+) & 35.87(+) & 21.64(+) & 1.29(+) & 10.96(+) & 0.10($\approx$) \\
 & Worst & 0.10 & 1.14 & 100.00 & 50.06 & 6.02 & 54.64 & 0.20 \\

Quartic & Best & 0.00 & 0.01 & 0.27 & 0.00 & 0.00 & 0.15 & 0.00 \\
 & Mean & \textbf{0.00} & 0.22(+) & 14.96(+) & 6.73(+) & 0.16(+) & 10.75(+) & 16.28(+) \\
 & Worst & 0.01 & 1.38 & 69.09 & 49.65 & 0.02 & 47.21 & 100.00 \\
 
Rastrigin & Best & 0.00 & 1.00 & 21.35 & 3.90 & 5.52 & 20.97 & 1.53 \\
 & Mean & \textbf{6.06} & 17.60(+) & 56.77(+) & 50.31(+) & 16.89(+) & 52.26(+) & 16.94(+) \\
 & Worst & 13.45 & 33.17 & 100.00 & 90.17 & 29.95 & 82.23 & 37.54 \\
 
Rosenbrock & Best & 0.00 & 0.02 & 0.95 & 0.63 & 0.00 & 0.30 & 0.01 \\
 & Mean & \textbf{0.03} & 0.90(+) & 29.75(+) & 8.39(+) & 1.21(+) & 21.74(+) & 4.57(+) \\
 & Worst & 0.06 & 3.34 & 88.45 & 28.83 & 5.34 & 100.00 & 17.46 \\
 
Scheffer & Best & 0.02 & 0.00 & 2.53 & 0.34 & 0.85 & 5.08 & 0.04 \\
 & Mean & \textbf{0.52} & 2.90(+) & 16.04(+) & 13.18(+) & 3.96(+) & 30.16(+) & 2.01(+) \\
 & Worst & 1.74 & 7.57 & 28.71 & 28.64 & 7.72 & 100.00 & 8.16 \\

SchwefelDouble & Best & 0.00 & 0.00 & 0.27 & 3.55 & 0.00 & 3.59 & 0.00 \\
 & Mean & \textbf{0.00} & \textbf{0.00}($\approx$) & 36.07(+) & 42.22(+) & 0.02(+) & 23.84(+) & 0.02(+) \\
 & Worst & 0.00 & 0.00 & 83.96 & 100.00 & 0.11 & 54.91 & 0.07 \\

SchwefelMax & Best & 0.00 & 1.04 & 23.46 & 1.27 & 0.19 & 9.46 & 0.31 \\
 & Mean & \textbf{0.10} & 21.70(+) & 62.46(+) & 42.18(+) & 3.52(+) & 40.62(+) & 0.81(+) \\
 & Worst & 0.21 & 42.48 & 100.00 & 75.10 & 7.16 & 57.53 & 1.76 \\

SchwefelSin & Best & 0.00 & 1.71 & 54.66 & 31.72 & 8.39 & 42.95 & 41.61 \\
 & Mean & \textbf{2.00} & 21.56(+) & 76.90(+) & 65.12(+) & 27.40(+) & 71.35(+) & 61.49(+) \\
 & Worst & 19.10 & 53.56 & 98.82 & 100.00 & 52.98 & 92.60 & 86.26 \\

Stairs & Best & 0.00 & 0.00 & 7.69 & 0.00 & 0.00 & 0.00 & 0.00 \\
 & Mean & \textbf{0.00} & \textbf{0.00}($\approx$) & 25.38(+) & 41.54(+) & \textbf{0.00}($\approx$) & 25.38(+) & 2.31(+) \\
 & Worst & 0.00 & 0.00 & 38.46 & 76.92 & 0.00 & 100.00 & 7.69 \\

Weierstrass & Best & 11.89 & 18.15 & 30.75 & 0.00 & 13.96 & 19.44 & 7.01 \\
 & Mean & \textbf{25.57} & 29.17($\approx$) & 64.84(+) & 38.57(+) & 27.34(-) & 57.88(+) & 28.36($\approx$) \\
 & Worst & 36.94 & 42.20 & 100.00 & 71.30 & 37.81 & 98.16 & 49.08 \\
\bottomrule
\end{tabular}
\end{table}
\begin{table}[]
\caption{Predicted mean optimal value over 30 trials for each method after 100 function evaluations for $20$ dimensional search space.}
\label{table:comparison_20d}
\begin{tabular}{llrrrrrrr}
\toprule
Function & Metric & MTRBO & EGO & TREGO & TRIKE & TURBO & TRLBO & EGP-TS \\
\midrule
Abs & Best & 0.00 & 66.56 & 73.06 & 69.34 & 57.17 & 58.40 & 12.22 \\
 & Mean & 2.15 & 70.45(+) & 87.51(+) & 88.66(+) & 75.45(+) & 82.79(+) & 34.21(+) \\
 & Worst & 4.38 & 75.78 & 98.40 & 100.00 & 88.43 & 99.89 & 59.28 \\
Ackley & Best & 14.07 & 53.62 & 71.18 & 70.11 & 65.27 & 67.30 & 0.00 \\
 & Mean & 16.90 & 65.74(+) & 83.87(+) & 87.48(+) & 82.29(+) & 84.04(+) & 22.96($\approx$) \\
 & Worst & 21.30 & 65.74 & 96.37 & 100.00 & 89.10 & 91.78 & 94.63 \\
AckleyTest & Best & 0.00 & 19.82 & 24.28 & 37.52 & 20.75 & 22.04 & 2.20 \\
 & Mean & 24.10 & 37.03(+) & 60.13(+) & 68.53(+) & 44.19(+) & 64.91(+) & 19.63(-) \\
 & Worst & 40.03 & 53.07 & 96.63 & 94.19 & 67.27 & 100.00 & 89.90 \\
Eggholder & Best & 0.00 & 47.99 & 57.42 & 67.45 & 50.92 & 44.56 & 35.02 \\
 & Mean & 17.86 & 59.91(+) & 72.89(+) & 81.22(+) & 60.47(+) & 74.13(+) & 61.89(+) \\
 & Worst & 30.73 & 69.62 & 86.42 & 100.00 & 75.34 & 92.77 & 75.59 \\
Griewank & Best & 3.32 & 13.88 & 29.82 & 15.25 & 9.30 & 39.19 & 0.00 \\
 & Mean & 17.15 & 25.85(+) & 58.56(+) & 51.90(+) & 37.01($\approx$) & 62.84(+) & 13.28(-) \\
 & Worst & 28.37 & 45.34 & 90.52 & 100.00 & 67.29 & 85.68 & 28.51 \\
Penalty2 & Best & 0.00 & 20.47 & 36.12 & 33.28 & 11.80 & 29.69 & 3.27 \\
 & Mean & 0.21 & 29.39(+) & 55.00(+) & 55.39(+) & 33.97(+) & 49.97(+) & 5.00(+) \\
 & Worst & 0.45 & 40.89 & 82.56 & 77.38 & 53.46 & 74.71 & 100.00 \\
Quartic & Best & 4.66 & 23.40 & 37.24 & 25.67 & 26.35 & 0.00 & 15.79 \\
 & Mean & 11.30 & 40.96(+) & 58.64(+) & 55.07(+) & 38.01(+) & 46.84(+) & 39.06(+) \\
 & Worst & 18.29 & 59.18 & 100.00 & 89.41 & 54.72 & 80.74 & 61.30 \\
Rastrigin & Best & 26.96 & 58.27 & 68.19 & 1.25 & 0.00 & 65.49 & 26.19 \\
 & Mean & 41.00 & 65.96(+) & 81.00(+) & 17.70(-) & 5.20(-) & 77.91(+) & 45.74(+) \\
 & Worst & 58.90 & 77.19 & 96.92 & 37.79 & 14.36 & 100.00 & 84.39 \\
Rosenbrock & Best & 0.00 & 3.05 & 33.86 & 47.61 & 45.58 & 38.57 & 6.81 \\
 & Mean & 1.15 & 5.14(+) & 66.06(+) & 61.81(+) & 59.71(+) & 62.33(+) & 19.29(+) \\
 & Worst & 6.58 & 7.74 & 93.45 & 80.69 & 81.78 & 100.00 & 39.95 \\
Scheffer & Best & 22.22 & 53.14 & 75.08 & 43.33 & 0.00 & 43.00 & 35.43 \\
 & Mean & 41.70 & 61.75(+) & 100.00(+) & 75.75(+) & 0.38(+) & 73.56(+) & 54.59(+) \\
 & Worst & 56.78 & 67.35 & 85.38 & 96.15 & 0.46 & 92.51 & 66.98 \\
SchwefelDouble & Best & 0.00 & 18.94 & 23.50 & 23.91 & 20.97 & 31.17 & 0.98 \\
 & Mean & 3.08 & 22.30(+) & 42.18(+) & 41.72(+) & 28.77(+) & 50.35(+) & 11.53($\approx$) \\
 & Worst & 5.30 & 27.28 & 63.86 & 100.00 & 42.91 & 93.06 & 22.90 \\
SchwefelMax & Best & 0.00 & 50.11 & 68.87 & 63.29 & 54.40 & 47.38 & 27.83 \\
 & Mean & 18.74 & 65.91(+) & 83.65(+) & 81.69(+) & 74.27(+) & 80.55(+) & 52.64(+) \\
 & Worst & 41.89 & 77.41 & 94.82 & 94.26 & 91.81 & 97.51 & 100.00 \\
SchwefelSin & Best & 40.51 & 59.73 & 77.75 & 68.23 & 46.18 & 0.00 & 69.56 \\
 & Mean & 45.00 & 67.33(+) & 88.11(+) & 79.64(+) & 63.97(+) & 82.40(+) & 83.90(+) \\
 & Worst & 47.76 & 75.18 & 100.00 & 93.99 & 73.58 & 92.35 & 95.45 \\
Stairs & Best & 0.00 & 46.76 & 41.67 & 46.76 & 49.07 & 48.61 & 1.85 \\
 & Mean & 1.85 & 58.94(+) & 65.79(+) & 68.61(+) & 62.08(+) & 72.31(+) & 4.35($\approx$) \\
 & Worst & 3.24 & 72.22 & 81.02 & 86.11 & 70.83 & 100.00 & 7.41 \\
Weierstrass & Best & 12.66 & 54.92 & 46.48 & 32.19 & 27.91 & 41.68 & 0.00 \\
 & Mean & 23.98 & 68.43(+) & 59.00(+) & 42.29(+) & 41.55(+) & 61.30(+) & 25.83(-) \\
 & Worst & 32.08 & 100.00 & 77.09 & 64.39 & 49.02 & 91.51 & 40.91 \\
\bottomrule
\end{tabular}
\end{table}
\begin{table}[]
\caption{Predicted mean optimal value over 30 trials for each method after 100 function evaluations for $50$ dimensional search space.}
\label{table:comparison_50d}
\begin{tabular}{llrrrrrrr}
\toprule
Function & Metric & MTRBO & EGO & TREGO & TRIKE & TURBO & TRLBO & EGP-TS \\
\midrule
Abs & Best & 0.00 & 64.45 & 86.25 & 82.48 & 71.82 & 72.25 & 88.36 \\
 & Mean & 5.12 & 77.65(+) & 90.16(+) & 90.69(+) & 81.34(+) & 89.37(+) & 91.09(+) \\
 & Worst & 15.26 & 84.57 & 95.39 & 99.08 & 90.74 & 100.00 & 95.94 \\
Ackley & Best & 0.00 & 65.29 & 62.22 & 51.97 & 65.29 & 69.42 & 71.76 \\
 & Mean & 4.58 & 71.55(+) & 82.88(+) & 80.84(+) & 77.90(+) & 80.21(+) & 77.27(+) \\
 & Worst & 10.39 & 75.89 & 100.00 & 96.96 & 85.10 & 94.90 & 80.74 \\
AckleyTest & Best & 0.00 & 38.43 & 62.54 & 53.37 & 41.72 & 51.30 & 12.44 \\
 & Mean & 24.36 & 49.76(+) & 81.67(+) & 74.53(+) & 62.03(+) & 76.08(+) & 32.02(-) \\
 & Worst & 54.67 & 69.20 & 100.00 & 91.84 & 78.16 & 88.71 & 51.86 \\
Eggholder & Best & 0.00 & 47.07 & 74.61 & 53.96 & 16.40 & 65.82 & 31.61 \\
 & Mean & 17.55 & 63.16(+) & 82.44(+) & 82.75(+) & 53.27(+) & 78.00(+) & 66.00(+) \\
 & Worst & 29.01 & 72.82 & 100.00 & 99.50 & 72.12 & 94.19 & 91.53 \\
Griewank & Best & 0.00 & 6.44 & 10.40 & 2.22 & 14.35 & 22.22 & 6.97 \\
 & Mean & 17.30 & 26.36(+) & 55.86(+) & 56.16(+) & 37.46(+) & 55.95(+) & 27.75(+) \\
 & Worst & 24.76 & 43.80 & 89.31 & 100.00 & 57.35 & 78.96 & 70.00 \\
Penalty2 & Best & 0.00 & 32.86 & 58.85 & 35.97 & 44.71 & 37.30 & 61.34 \\
 & Mean & 1.48 & 46.00(+) & 65.74(+) & 58.88(+) & 53.93(+) & 58.54(+) & 81.61(+) \\
 & Worst & 3.95 & 55.57 & 74.93 & 71.30 & 64.21 & 76.87 & 100.00 \\
Quartic & Best & 0.00 & 26.98 & 48.46 & 49.46 & 16.38 & 37.67 & 7.76 \\
 & Mean & 10.04 & 44.46(+) & 70.36(+) & 62.12(+) & 54.98(+) & 62.63(+) & 26.86(+) \\
 & Worst & 18.99 & 60.17 & 100.00 & 75.63 & 84.25 & 78.25 & 42.22 \\
Rastrigin & Best & 25.28 & 72.55 & 81.71 & 3.12 & 0.00 & 78.47 & 25.34 \\
 & Mean & 45.56 & 80.33(+) & 90.79(+) & 15.44(+) & 9.99(+) & 89.44(+) & 35.78(+) \\
 & Worst & 66.67 & 84.23 & 100.00 & 24.63 & 15.65 & 96.70 & 42.45 \\
Rosenbrock & Best & 0.00 & 30.13 & 41.90 & 32.93 & 37.56 & 30.25 & 21.68 \\
 & Mean & 5.16 & 40.74(+) & 63.16(+) & 58.72(+) & 51.61(+) & 59.69(+) & 26.36(+) \\
 & Worst & 9.99 & 54.48 & 100.00 & 93.51 & 73.63 & 89.05 & 49.19 \\
Scheffer & Best & 20.27 & 25.03 & 48.48 & 0.00 & 27.36 & 35.58 & 37.73 \\
 & Mean & 27.96 & 43.06(+) & 74.29(+) & 51.00(+) & 62.93(+) & 63.81(+) & 46.77(+) \\
 & Worst & 35.46 & 53.70 & 100.00 & 72.54 & 90.40 & 91.34 & 82.61 \\
SchwefelDouble & Best & 0.00 & 7.71 & 12.23 & 15.65 & 9.22 & 11.03 & 3.76 \\
 & Mean & 0.86 & 12.06(+) & 49.22(+) & 43.89(+) & 17.81(+) & 27.00(+) & 6.28(+) \\
 & Worst & 2.05 & 15.12 & 100.00 & 99.91 & 31.59 & 72.38 & 12.96 \\
SchwefelMax & Best & 0.00 & 50.23 & 47.96 & 51.10 & 61.63 & 52.34 & 82.75 \\
 & Mean & 7.96 & 56.35(+) & 70.97(+) & 73.24(+) & 72.62(+) & 72.76(+) & 90.13(+) \\
 & Worst & 17.50 & 65.50 & 81.41 & 82.64 & 81.44 & 81.18 & 100.00 \\
SchwefelSin & Best & 0.00 & 38.06 & 39.53 & 40.31 & 32.90 & 32.42 & 42.25 \\
 & Mean & 8.28 & 46.16(+) & 64.86(+) & 69.07(+) & 42.98(+) & 66.58(+) & 64.73(+) \\
 & Worst & 17.65 & 60.40 & 86.35 & 100.00 & 51.21 & 96.05 & 86.30 \\
Stairs & Best & 0.00 & 50.12 & 71.13 & 66.97 & 48.04 & 59.12 & 39.95 \\
 & Mean & 6.51 & 64.43(+) & 82.33(+) & 76.86(+) & 66.95(+) & 73.12(+) & 64.18(+) \\
 & Worst & 16.63 & 69.98 & 100.00 & 97.46 & 86.14 & 86.61 & 89.61 \\
Weierstrass & Best & 18.93 & 46.71 & 36.53 & 3.21 & 0.00 & 42.13 & 7.85 \\
 & Mean & 24.36 & 66.20(+) & 70.70(+) & 29.31(+) & 26.75(-) & 61.44(+) & 16.33(-) \\
 & Worst & 31.09 & 86.56 & 100.00 & 56.81 & 49.78 & 84.62 & 26.87 \\
\bottomrule
\end{tabular}
\end{table}

\begin{table}[]
\caption{Predicted mean optimal value over 30 trials for each method after 100 function evaluations for $100$ dimensional search space.}
\label{table:comparison_100d}
\begin{tabular}{llrrrrrrr}
\toprule
Function & Metric & MTRBO & EGO & TREGO & TRIKE & TURBO & TRLBO & EGP-TS \\
\midrule
Abs & Best & 0.00 & 63.04 & 81.44 & 73.25 & 76.01 & 77.22 & 76.04 \\
 & Mean & 6.56 & 76.20(+) & 88.42(+) & 85.42(+) & 80.50(+) & 85.45(+) & 80.82(+) \\
 & Worst & 15.80 & 82.15 & 100.00 & 94.30 & 84.48 & 95.20 & 96.66 \\
Ackley & Best & 0.00 & 57.78 & 68.09 & 61.49 & 67.39 & 59.41 & 60.41 \\
 & Mean & 7.18 & 63.82(+) & 74.65(+) & 74.15(+) & 71.04(+) & 73.28(+) & 77.34(+) \\
 & Worst & 15.55 & 68.03 & 78.68 & 83.85 & 79.76 & 85.25 & 100.00 \\
AckleyTest & Best & 0.00 & 36.77 & 42.84 & 46.09 & 39.93 & 40.34 & 12.32 \\
 & Mean & 21.78 & 44.46(+) & 61.38(+) & 62.06(+) & 50.83(+) & 56.92(+) & 31.95($\approx$) \\
 & Worst & 35.59 & 53.04 & 83.36 & 76.64 & 62.96 & 70.44 & 100.00 \\
Eggholder & Best & 0.00 & 39.28 & 56.83 & 60.17 & 27.47 & 47.45 & 88.13 \\
 & Mean & 12.62 & 56.18(+) & 77.44(+) & 76.34(+) & 47.83(+) & 68.67(+) & 82.46(+) \\
 & Worst & 42.93 & 69.40 & 100.00 & 92.06 & 61.90 & 85.01 & 69.17 \\
Griewank & Best & 0.00 & 13.71 & 30.33 & 33.72 & 16.92 & 31.22 & 38.44 \\
 & Mean & 16.94 & 28.30($\approx$) & 61.62(+) & 49.01(+) & 35.71(+) & 54.56(+) & 57.65(+) \\
 & Worst & 28.97 & 40.70 & 99.01 & 69.99 & 52.73 & 79.01 & 100.00 \\
Penalty2 & Best & 0.00 & 5.25 & 6.62 & 5.86 & 5.66 & 6.19 & 7.51 \\
 & Mean & 0.62 & 6.28(+) & 7.41(+) & 7.60(+) & 7.37(+) & 7.84(+) & 8.14(+) \\
 & Worst & 1.42 & 6.94 & 8.36 & 8.47 & 8.12 & 9.07 & 100.00 \\
Quartic & Best & 0.00 & 31.78 & 34.13 & 35.26 & 30.91 & 17.63 & 15.55 \\
 & Mean & 6.43 & 50.97(+) & 67.21(+) & 54.16(+) & 57.22(+) & 45.60(+) & 27.94(+) \\
 & Worst & 11.86 & 64.02 & 100.00 & 65.26 & 74.67 & 82.23 & 63.70 \\
Rastrigin & Best & 0.00 & 31.30 & 62.17 & 50.33 & 40.23 & 39.27 & 12.73 \\
 & Mean & 14.78 & 50.76(+) & 79.28(+) & 61.79(+) & 62.01(+) & 63.18(+) & 30.37(+) \\
 & Worst & 33.32 & 60.03 & 100.00 & 86.88 & 73.36 & 83.32 & 64.96 \\
Rosenbrock & Best & 0.00 & 51.33 & 65.94 & 29.59 & 57.32 & 47.08 & 24.52 \\
 & Mean & 7.71 & 57.12(+) & 76.12(+) & 54.93(+) & 65.05(+) & 63.03(+) & 54.60(+) \\
 & Worst & 16.42 & 64.34 & 90.00 & 75.24 & 73.20 & 82.50 & 100.00 \\
Scheffer & Best & 35.38 & 38.81 & 56.78 & 50.74 & 46.77 & 52.32 & 0.00 \\
 & Mean & 45.30 & 61.78(+) & 84.14(+) & 78.34(+) & 63.00(+) & 72.22(+) & 59.90(+) \\
 & Worst & 51.12 & 77.79 & 100.00 & 96.46 & 70.29 & 80.81 & 72.96 \\
SchwefelDouble & Best & 0.00 & 4.75 & 29.77 & 21.75 & 12.18 & 16.41 & 17.75 \\
 & Mean & 4.20 & 25.01(+) & 54.63(+) & 45.32(+) & 29.30(+) & 46.77(+) & 38.25(+) \\
 & Worst & 8.34 & 49.21 & 100.00 & 97.07 & 59.55 & 81.25 & 72.66 \\
SchwefelMax & Best & 0.00 & 66.43 & 70.77 & 53.83 & 70.13 & 79.36 & 61.46 \\
 & Mean & 29.06 & 76.38(+) & 85.55(+) & 87.99(+) & 85.38(+) & 89.99(+) & 80.62(+) \\
 & Worst & 58.93 & 85.29 & 98.26 & 99.90 & 97.26 & 97.88 & 100.00 \\
SchwefelSin & Best & 0.00 & 59.79 & 66.52 & 71.16 & 37.26 & 54.15 & 60.98 \\
 & Mean & 26.00 & 66.75(+) & 82.54(+) & 100.00(+) & 58.41(+) & 79.27(+) & 83.30(+) \\
 & Worst & 42.45 & 71.87 & 93.30 & 86.17 & 78.12 & 94.97 & 95.17 \\
Stairs & Best & 0.00 & 64.90 & 70.67 & 68.27 & 60.10 & 62.98 & 57.33 \\
 & Mean & 20.26 & 71.26(+) & 83.77(+) & 84.10(+) & 70.94(+) & 79.01(+) & 64.87(+) \\
 & Worst & 31.13 & 77.16 & 100.00 & 93.87 & 77.40 & 94.23 & 77.52 \\
Weierstrass & Best & 0.00 & 28.53 & 45.49 & 40.43 & 56.08 & 9.57 & 23.93 \\
 & Mean & 20.95 & 69.82(+) & 75.96(+) & 60.62(+) & 71.60(+) & 41.09($\approx$) & 43.18(+) \\
 & Worst & 32.50 & 87.91 & 100.00 & 86.09 & 95.53 & 72.94 & 75.90 \\
\bottomrule
\end{tabular}
\end{table}


\begin{table}[]
\caption{Predicted mean optimal value over 30 trials for each method after 100 function evaluations for $500$ dimensional search space.}
\label{table:comparison_500d}
\begin{tabular}{llrrrrrrr}
\toprule
Function & Metric & MTRBO & EGO & TREGO & TRIKE & TURBO & TRLBO & EGP-TS \\
\midrule
Abs & Best & 0.00 & 58.37 & 70.49 & 63.92 & 64.01 & 70.23 & 53.96 \\
 & Mean & 11.90 & 63.74(+) & 76.35(+) & 73.68(+) & 69.41(+) & 74.88(+) & 73.97(+) \\
 & Worst & 20.33 & 68.74 & 82.01 & 83.55 & 74.90 & 80.55 & 100.00 \\
Ackley & Best & 0.00 & 53.38 & 54.92 & 54.96 & 58.12 & 55.41 & 70.13 \\
 & Mean & 3.41 & 57.88(+) & 64.66(+) & 65.11(+) & 62.27(+) & 59.70(+) & 79.28(+) \\
 & Worst & 8.64 & 60.63 & 72.26 & 73.23 & 69.28 & 63.32 & 100.00 \\
AckleyTest & Best & 0.00 & 18.17 & 35.94 & 31.38 & 22.62 & 28.63 & 4.21 \\
 & Mean & 14.16 & 25.85(+) & 46.76(+) & 46.63(+) & 33.55(+) & 41.91(+) & 36.84($\approx$) \\
 & Worst & 26.04 & 31.53 & 63.18 & 60.60 & 46.96 & 57.89 & 100.00 \\
Eggholder & Best & 0.00 & 55.34 & 70.44 & 57.49 & 41.14 & 55.45 & 70.84 \\
 & Mean & 12.53 & 63.66(+) & 83.05(+) & 77.69(+) & 53.84(+) & 71.56(+) & 73.19(+) \\
 & Worst & 23.52 & 69.30 & 100.00 & 91.17 & 65.49 & 89.95 & 78.29 \\
Griewank & Best & 0.00 & 14.91 & 18.82 & 0.32 & 10.42 & 15.40 & 21.42 \\
 & Mean & 8.17 & 19.56(+) & 40.35(+) & 38.81(+) & 29.14(+) & 40.71(+) & 37.03(+) \\
 & Worst & 15.17 & 30.56 & 60.59 & 60.82 & 48.98 & 66.54 & 100.00 \\
Penalty2 & Best & 0.00 & 56.56 & 76.62 & 61.75 & 65.35 & 73.30 & 22.94 \\
 & Mean & 2.80 & 69.26(+) & 86.38(+) & 78.16(+) & 73.46(+) & 84.80(+) & 35.51(+) \\
 & Worst & 9.63 & 81.42 & 97.45 & 98.82 & 84.73 & 100.00 & 61.07 \\
Quartic & Best & 0.00 & 61.66 & 69.11 & 2.54 & 64.14 & 19.12 & 37.00 \\
 & Mean & 8.87 & 70.50(+) & 85.73(+) & 24.06(+) & 74.78(+) & 37.03(+) & 52.24(+) \\
 & Worst & 17.48 & 80.55 & 100.00 & 84.74 & 82.20 & 63.85 & 63.77 \\
Rastrigin & Best & 13.77 & 56.47 & 61.12 & 0.00 & 59.43 & 62.20 & 31.50 \\
 & Mean & 21.06 & 63.93(+) & 76.46(+) & 34.96(+) & 70.62(+) & 70.07(+) & 83.83(+) \\
 & Worst & 34.89 & 76.06 & 94.25 & 96.78 & 79.77 & 79.48 & 100.00 \\
Rosenbrock & Best & 0.00 & 67.79 & 70.45 & 40.87 & 76.13 & 40.00 & 42.40 \\
 & Mean & 6.36 & 76.09(+) & 90.46(+) & 69.57(+) & 82.09(+) & 61.50(+) & 60.44(+) \\
 & Worst & 12.94 & 81.30 & 100.00 & 85.83 & 86.04 & 82.58 & 76.25 \\
Scheffer & Best & 0.00 & 30.89 & 50.56 & 21.07 & 15.55 & 20.73 & 23.62 \\
 & Mean & 19.47 & 53.22(+) & 75.95(+) & 48.52(+) & 66.35(+) & 49.94(+) & 52.28(+) \\
 & Worst & 34.62 & 69.77 & 96.24 & 76.63 & 95.30 & 78.72 & 100.00 \\
SchwefelDouble & Best & 0.00 & 4.37 & 16.53 & 6.64 & 13.43 & 17.46 & 7.89 \\
 & Mean & 1.19 & 15.22(+) & 48.03(+) & 34.69(+) & 39.26(+) & 33.96(+) & 40.91(+) \\
 & Worst & 2.65 & 24.39 & 96.79 & 87.63 & 24.38 & 50.08 & 100.00 \\
SchwefelMax & Best & 0.00 & 73.61 & 81.85 & 51.28 & 84.92 & 74.48 & 35.43 \\
 & Mean & 42.57 & 78.01(+) & 92.44(+) & 87.84(+) & 92.91(+) & 84.98(+) & 62.89(+) \\
 & Worst & 71.41 & 84.52 & 99.51 & 98.41 & 100.00 & 97.86 & 90.14 \\
SchwefelSin & Best & 0.00 & 35.41 & 51.51 & 33.09 & 30.09 & 42.89 & 26.46 \\
 & Mean & 6.74 & 50.76(+) & 66.32(+) & 66.20(+) & 45.96(+) & 64.11(+) & 73.02(+) \\
 & Worst & 11.31 & 64.74 & 75.86 & 94.40 & 63.80 & 89.56 & 100.00 \\
Stairs & Best & 0.00 & 67.99 & 79.22 & 71.04 & 62.61 & 75.26 & 67.46 \\
 & Mean & 13.74 & 77.30(+) & 92.81(+) & 87.09(+) & 78.87(+) & 84.02(+) & 78.07(+) \\
 & Worst & 23.31 & 85.65 & 100.00 & 99.37 & 90.35 & 90.77 & 91.61 \\
Weierstrass & Best & 7.61 & 13.58 & 14.86 & 13.73 & 16.44 & 0.00 & 25.34 \\
 & Mean & 9.96 & 17.09(+) & 25.88(+) & 20.38(+) & 20.56(+) & 9.99(-) & 63.89(+) \\
 & Worst & 11.59 & 20.70 & 33.62 & 27.69 & 23.83 & 18.29 & 100.00 \\
\bottomrule
\end{tabular}
\end{table}

\begin{table}[htbp]
\caption{Comparison of performance measures on the portfolio optimization problem using data from the Indian National Stock Exchange (100 stocks) and the New York Stock Exchange (200 stocks).}
\label{table:portfolio}
\begin{tabular}{@{}|l|l|lllllll|@{}}
\toprule
Data 
& Performance        & MTRBO       & EGO        & TREGO       & TRIKE       & TURBO        & TRLBO   & EGP-TS  \\
\midrule
\multirow{4}{*}{Indian data}
& Return           &  0.22344    & 0.20196     & 0.22344     & 0.13709     & 0.21015     & 0.13071  & 0.14969   \\
& Risk             &  0.03835    & 0.06346     & 0.05621     & 0.02694     & 0.06691     & 0.03826  & 0.03518  \\
& Objective value& \textbf{0.204265} & 0.17023 & 0.19533     & 0.12362     & 0.176695     & 0.11157   & 0.1321  \\
& Sharpe ratio     &  \textbf{5.82549}    & 3.18257     & 3.97502     & 5.08827     & 3.14097     & 3.41652   & 4.25497  \\
\hline
\multirow{4}{*}{USA data }
& Return           &  0.18029    & 0.03769     & 0.04238     & 0.30064     & 0.13267     & 0.15505   & 0.11245  \\
& Risk             &  0.05524    & 0.01989     & 0.02261     & 0.11900     & 0.05422     & 0.06387  & 0.04159  \\
& Objective value& 0.15267 & 0.027745 & 0.031075     & \textbf{0.24114}   & 0.10556   & 0.123115   & 0.09165  \\
& Sharpe ratio     &  \textbf{3.26358}    & 1.89471     & 1.87401     & 2.52641     & 2.44675     & 2.42754  & 2.70377   \\
\botrule
\end{tabular}
\end{table}
\begin{table}[]
\caption{Sensitivity analysis of the hyperparameters for the proposed MTRBO method for Ackley function}
\label{table:sensitivity_analysis}
\begin{tabular}{llrrrr}
\toprule
No of initial & Initial radius & Length scale & No of iteration & No of subiteration & Objective value \\
\midrule
10 & 0.1 & 0.4 & 50 & 15 & 0.000002 \\
10 & 0.1 & 0.4 & 75 & 5 & 0.000981 \\
10 & 0.1 & 0.6 & 75 & 10 & 0.001584 \\
10 & 0.1 & 0.6 & 100 & 10 & 0.000001 \\
10 & 0.1 & 0.8 & 50 & 5 & 0.000023 \\
10 & 0.3 & 0.6 & 50 & 15 & 0.001330 \\
10 & 0.3 & 0.6 & 75 & 5 & 0.000002 \\
10 & 0.3 & 0.8 & 75 & 10 & 0.000004 \\
10 & 0.3 & 0.8 & 75 & 15 & 0.000427 \\
10 & 0.3 & 0.8 & 100 & 5 & 0.000198 \\
10 & 0.5 & 1.0 & 100 & 10 & 0.157913 \\
10 & 0.5 & 0.6 & 100 & 15 & 0.394116 \\
10 & 0.5 & 0.6 & 75 & 10 & 0.001505 \\
10 & 0.5 & 0.8 & 100 & 15 & 0.108247 \\
10 & 0.5 & 1.0 & 50 & 5 & 0.462274 \\
15 & 0.1 & 0.4 & 75 & 15 & 0.000054 \\
15 & 0.1 & 0.4 & 100 & 5 & 0.000000 \\
15 & 0.1 & 0.6 & 50 & 10 & 0.000004 \\
15 & 0.1 & 0.8 & 50 & 5 & 0.000985 \\
15 & 0.3 & 0.2 & 75 & 10 & 0.001858 \\
15 & 0.3 & 0.4 & 100 & 15 & 0.008548 \\
15 & 0.3 & 1.0 & 200 & 5 & 0.003710 \\
15 & 0.5 & 0.6 & 100 & 10 & 0.004772 \\
15 & 0.5 & 0.8 & 100 & 10 & 0.004534 \\
15 & 0.5 & 1 & 75 & 15 & 0.042851 \\
15 & 0.8 & 1.2 & 125 & 5 & 0.117522 \\
15 & 1.5 & 0.6 & 75 & 10 & 0.562320 \\
20 & 0.1 & 0.4 & 50 & 5 & 0.000001 \\
20 & 0.1 & 0.4 & 50 & 10 & 0.000000 \\
20 & 0.1 & 0.4 & 150 & 5 & 0.000000 \\
20 & 0.1 & 0.8 & 125 & 15 & 0.000002 \\
20 & 0.3 & 0.6 & 125 & 5 & 0.014945 \\
20 & 0.3 & 1.2 & 50 & 15 & 0.000005 \\
20 & 0.5 & 0.4 & 50 & 15 & 0.003265 \\
20 & 0.5 & 1.0 & 150 & 15 & 0.001013 \\
20 & 0.5 & 1.2 & 75 & 5 & 0.000060 \\
20 & 1.0 & 0.8 & 50 & 15 & 0.450452 \\
25 & 0.1 & 0.4 & 100 & 15 & 0.000000 \\
25 & 0.1 & 0.6 & 100 & 15 & 0.000116 \\
25 & 0.1 & 0.6 & 150 & 5 & 0.008211 \\
25 & 0.1 & 1.0 & 200 & 5 & 0.000001 \\
25 & 0.3 & 0.6 & 100 & 15 & 0.000000 \\
25 & 0.3 & 1.2 & 75 & 15 & 0.000002 \\
25 & 0.3 & 1.2 & 150 & 10 & 0.000022 \\
25 & 0.5 & 0.4 & 50 & 15 & 0.004026 \\
25 & 0.5 & 0.8 & 100 & 10 & 0.063435 \\
25 & 1.0 & 1.0 & 125 & 5 & 0.001457 \\
\bottomrule
\end{tabular}

\end{table}

\begin{figure}[]
\centering
\includegraphics[width=1\textwidth]{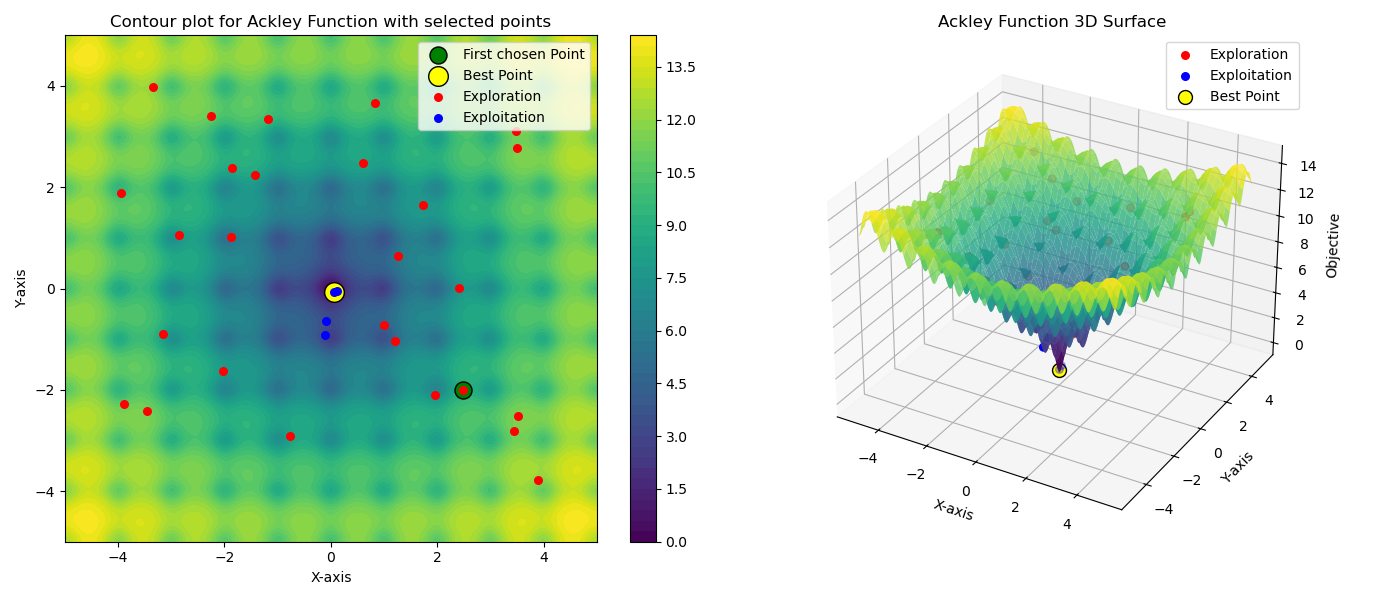}
\includegraphics[width=1\textwidth]{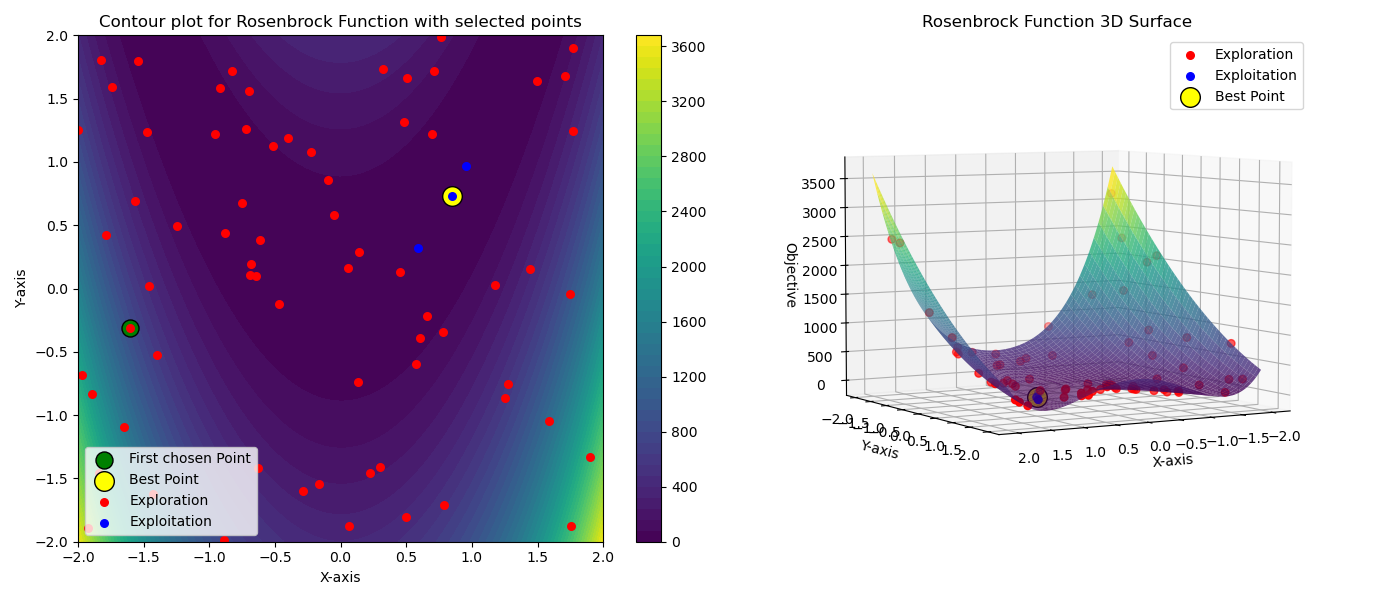}
\includegraphics[width=1\textwidth]{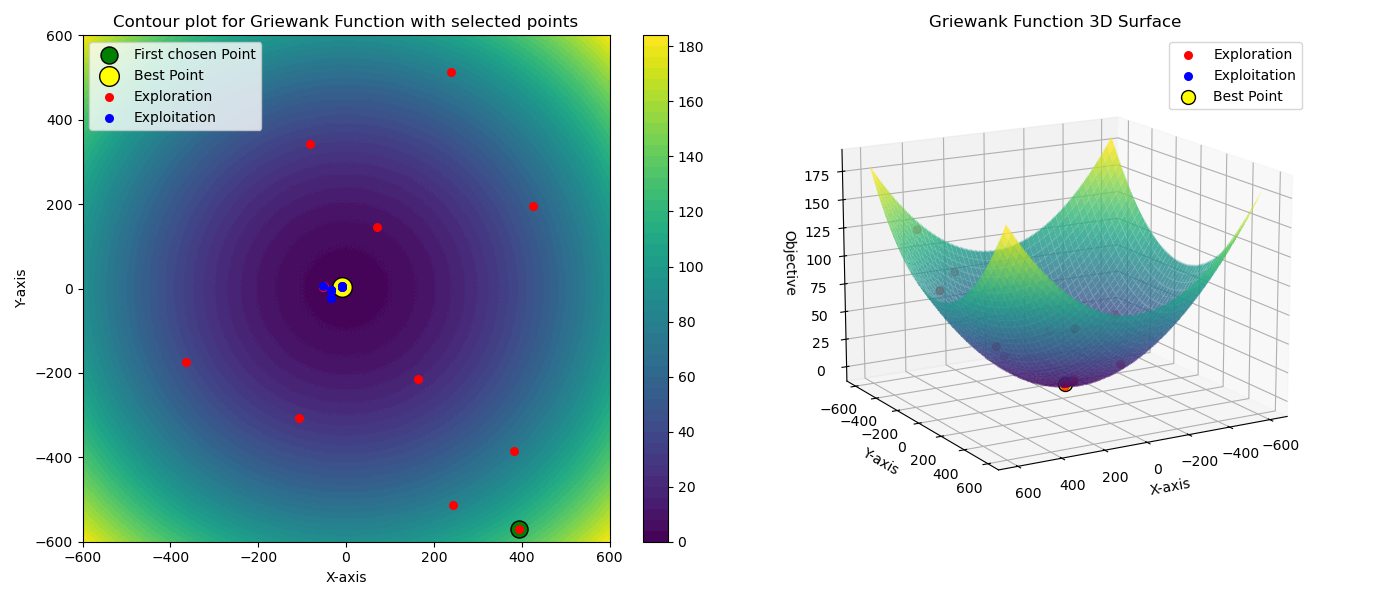}

\caption{Contour plot together with locations of the query points in the search space after every iteration of the optimization process and 3d surface plot for Ackley, Rosenbrock, and Griewank function for an arbitrary trial run. Points colored red are found from the exploration stage, and the points colored blue are found from the exploitation stage.}
\label{image:contour_1}
\end{figure}

\begin{figure}[htbp]
\centering
\includegraphics[width=1\textwidth]{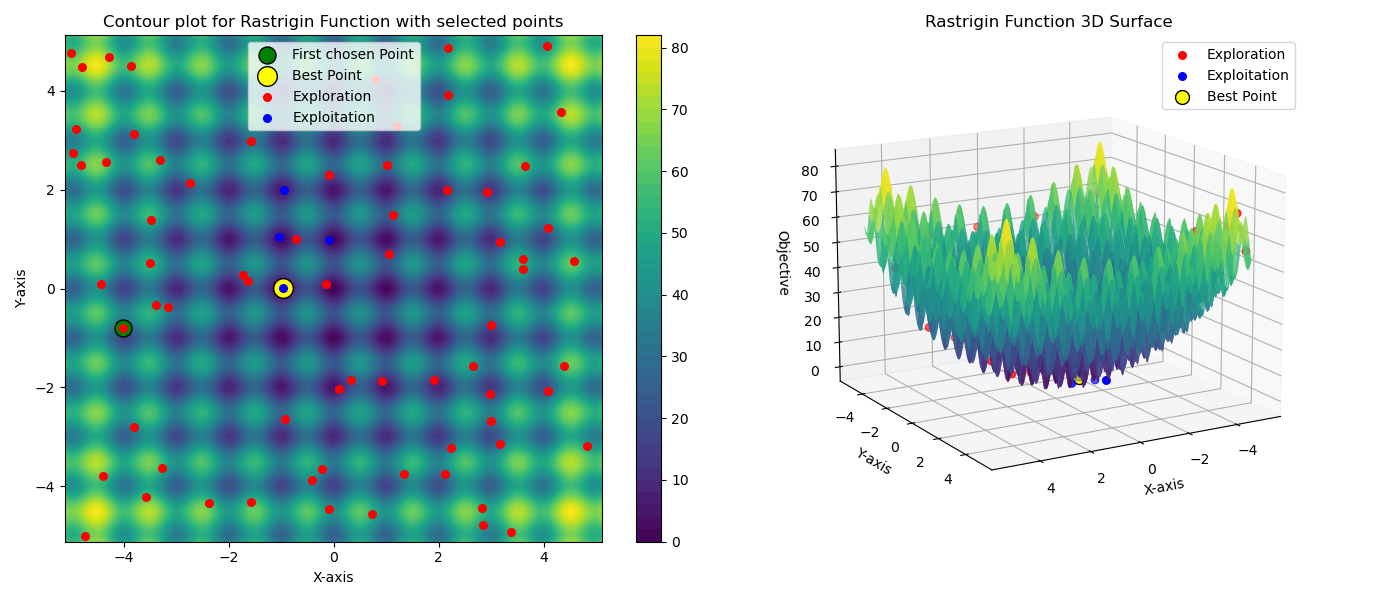}
\includegraphics[width=1\textwidth]{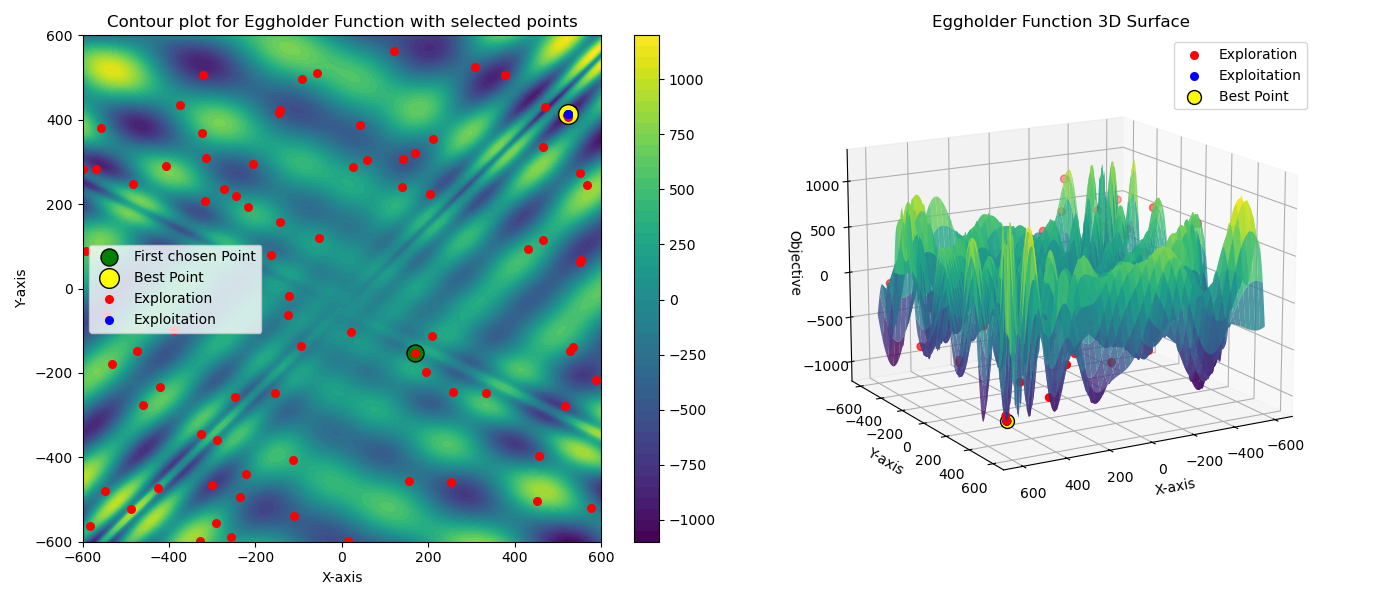}
\includegraphics[width=1\textwidth]{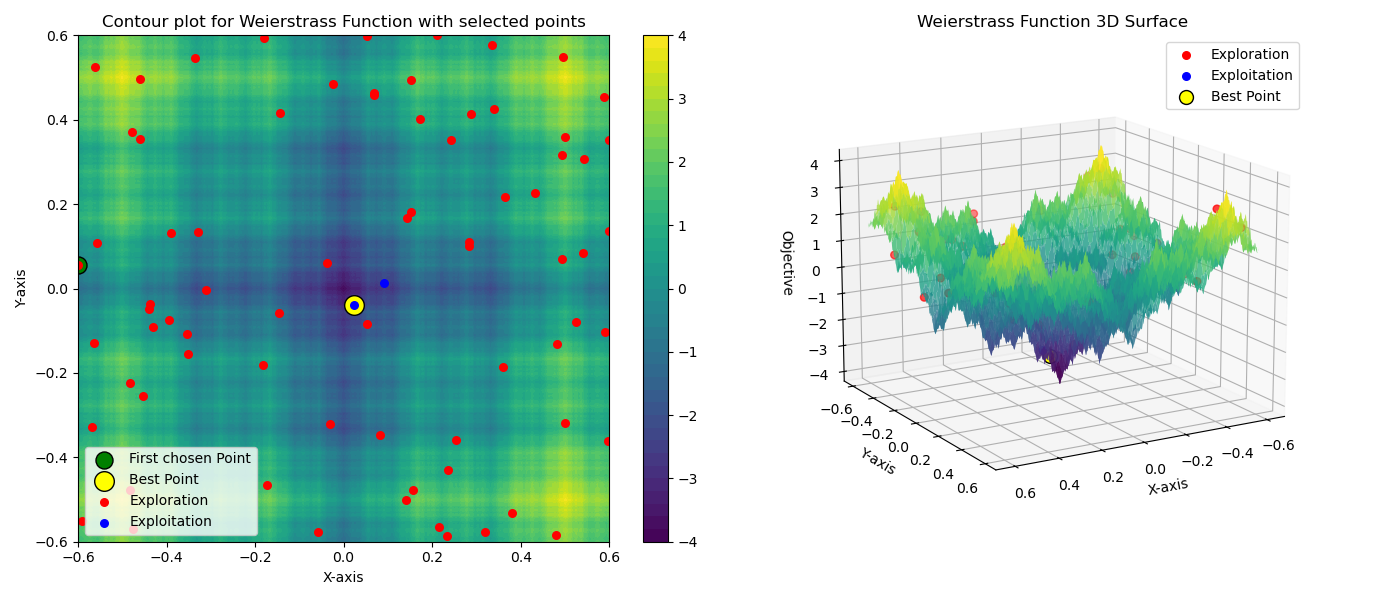}

\caption{Contour plot together with locations of the query points in the search space after every iteration of the optimization process and 3d surface plot for Rastrigin, Eggholder, and Weierstras function for an arbitrary trial run. Points colored red are found from the exploration stage, and the points colored blue are found from the exploitation stage.}
\label{image:contour_2}
\end{figure}

\begin{figure}[htbp]
\centering
\includegraphics[width=0.8\textwidth]{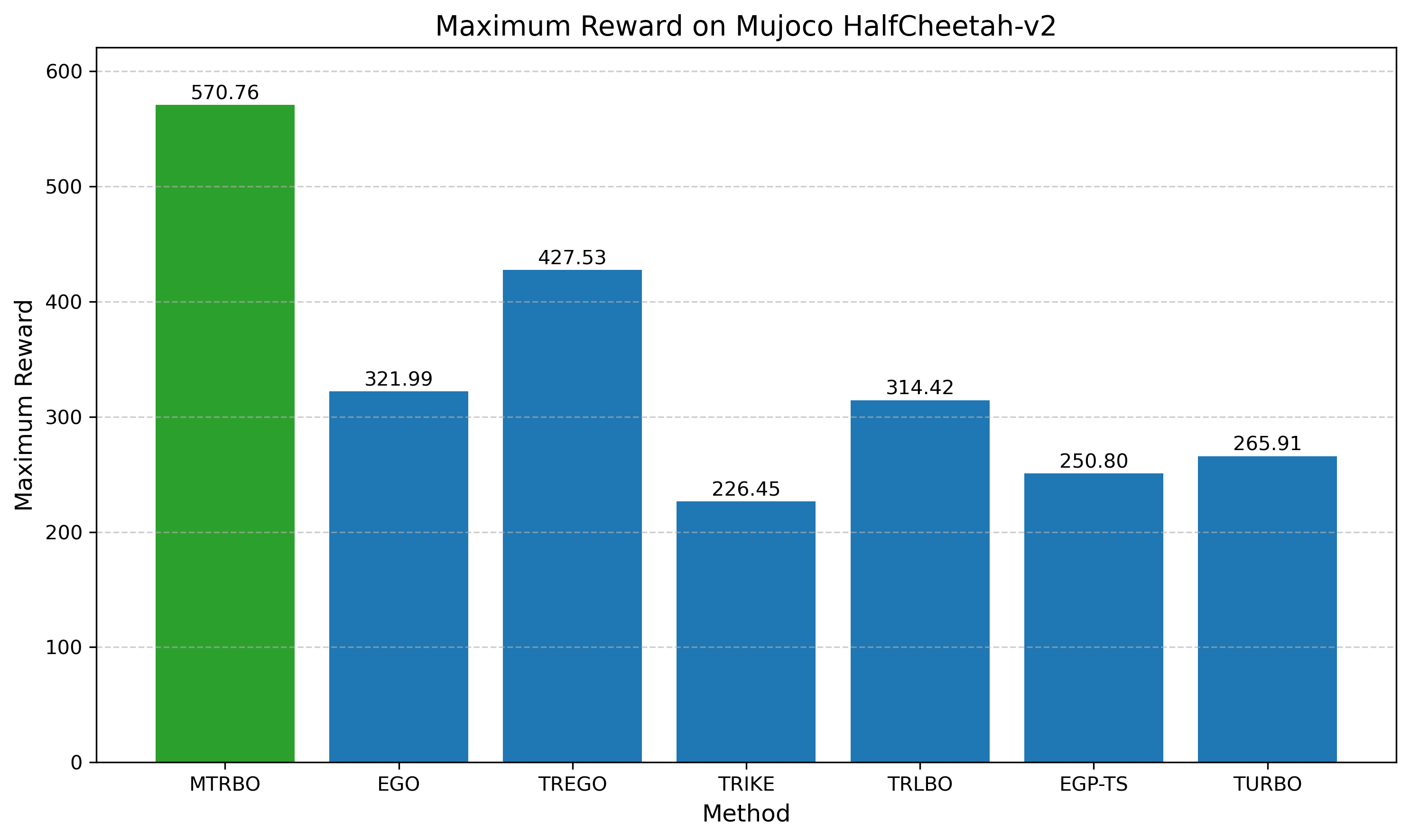}
\caption{Bar plot showing the performance comparison of MTRBO with other baseline methods on the MuJoCo HalfCheetah environment. The y-axis represents the cumulative reward, and the x-axis denotes the methods.}
\label{image:mujoco}
\end{figure}

\begin{figure}[htbp]
\centering

\includegraphics[width=0.8\textwidth]{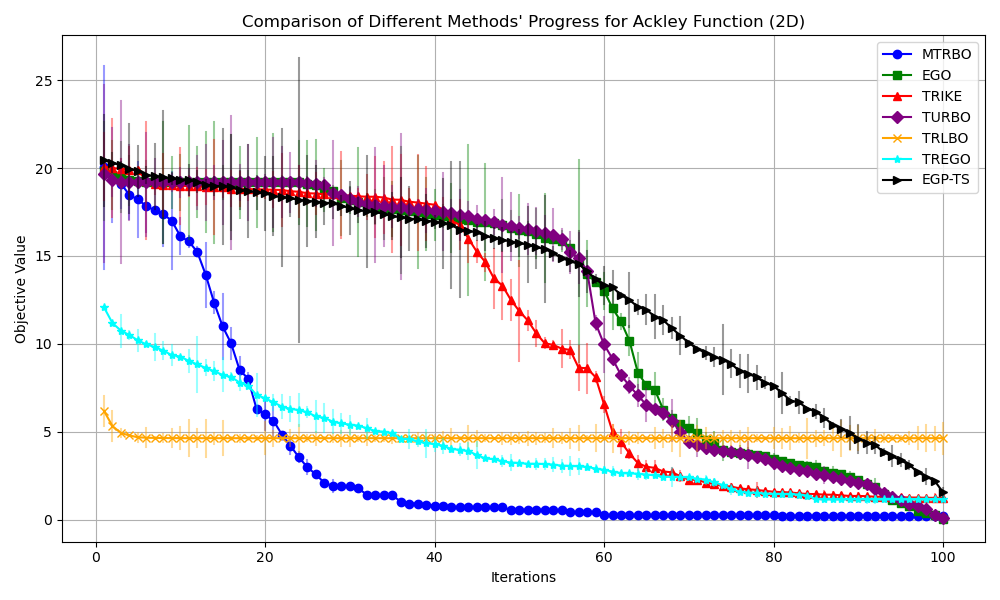}

\caption{Comparison of mean objective values with deviation from mean of all the methods for Ackley function (2D) against iteration.}
\label{image:progress}
\end{figure}

\begin{figure}[htbp]
\centering

\includegraphics[width=0.8\textwidth]{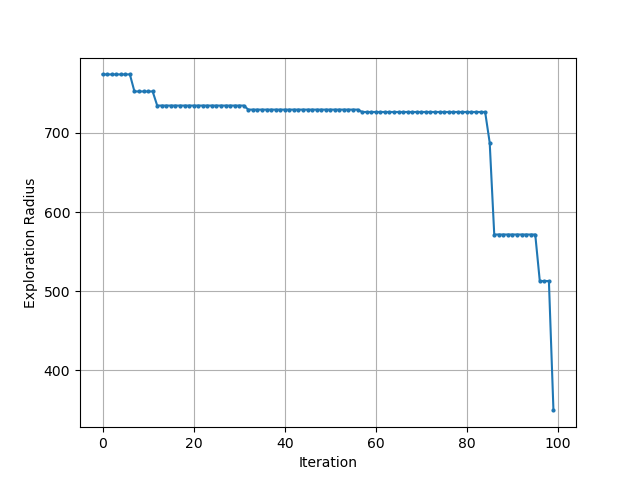}

\caption{Visualization of changes of the trust region radius for exploration stage after each iteration for Eggholder function (2D) .}
\label{image:exploration_radius}
\end{figure}

\subsection{Discussion}
To assess the effectiveness of the proposed algorithm, a comprehensive comparison against several state-of-the-art black-box optimization methods across a diverse set of test functions with varying dimensions has been conducted. The normalized scores of each method for all the test functions for different dimensions are presented in Tables \ref{table:comparison_2d}, \ref{table:comparison_20d}, \ref{table:comparison_50d}, \ref{table:comparison_100d}, \ref{table:comparison_500d} in detail.
\begin{itemize}
  \item For most test functions across different dimensions, the mean normalized score (as defined in Table \ref{table:experiment_summary}) achieved by the proposed MTRBO method is better than those of the other methods.
  \item The worst performance of the method is also close to the overall minimum observed across all methods and runs.
  \item In many cases, the minimum value obtained by the proposed method matches $f_{min}$.
\end{itemize}
\par
To understand the individual contributions of key components within the proposed method (i.e, exploration and exploitation), an ablation study is conducted. This involved systematically disabling one component of the algorithm and observing their impact on performance. Removing exploration leads to poor global search, and removing exploitation leads to overlook potential region. Enabling both exploration and exploitation in MTRBO leads to better results than using either exploration or exploitation alone. The results can be seen in details in Table \ref{table:ablation_study}.
\par
To validate the observed performance differences between MTRBO and other baseline methods, the Wilcoxon rank-sum test has been employed, as different random initial samples are used for different methods. This test was conducted for each benchmark function to determine whether the improvements achieved by MTRBO are statistically significant. The results are also shown in Tables (\ref{table:comparison_2d},\ref{table:comparison_20d},\ref{table:comparison_50d},\ref{table:comparison_100d},\ref{table:comparison_500d}), indicate that in the majority of cases, MTRBO significantly outperforms the compared methods at the $5\%$ significance level $(p<0.05)$. Symbols $+, -$ indicate compared method is significantly worse(larger objective values), significantly better(smaller objective values) with statistical significance $(p < 0.05)$, and $\approx$ indicates no significant difference($p \geq 0.05$) relative to MTRBO; similar to \citep{namura2021surrogate}, \citep{chugh2016surrogate}. These findings highlight that the improvements in performance are not due to random variation but reflect meaningful enhancements provided by the method.
\par
The proposed MTRBO method has been applied to two significantly different problem domains: the MuJoCo HalfCheetah reinforcement learning environment and a portfolio allocation problem in finance. In the MuJoCo HalfCheetah task, where the goal is to learn a control policy that maximizes the cumulative reward over time, the performance of MTRBO was benchmarked against several baseline optimization methods. As illustrated in the bar plot (Figure~\ref{image:mujoco}), it achieved a higher cumulative reward than the competing methods. In the portfolio allocation problem, the goal is to maximize return and minimize risk, under constraints on asset weights. Though in case of test functions normalization of each variable is not performed, but for most of the real world problems one needs to normalize each of the variables, as for different variables the ranges can differ significantly, which may affect the optimization performance. Here, in portfolio optimization problem, the weights are assumed to be in the range $[0,1]$. The comparative results presented in Table~\ref{table:portfolio} show that the proposed method most of the time achieves both higher objective values and better Sharpe ratios compared to other methods. 
\par
In Figure \ref{image:progress}, the progress of each method on the 2D Ackley function over 100 iterations is shown based on the mean values with the spreads from 30 independent trial runs to provide insight into how different methods approach an optimal or near-optimal solution, as well as their respective convergence speeds. From the figure, it is evident that the proposed MTRBO method exhibits a faster convergence rate compared to the baseline methods. Although TRLBO performs better than MTRBO in the initial stages, it tends to get stuck in a local optimum. This is primarily because TRLBO focuses heavily on exploitation, and also by employing trust regions to limit the number of samples used in the Gaussian process. While this enhances exploitation, it can lead the method to overlook potentially better regions for exploration. Similar to many Bayesian optimization methods, MTRBO emphasizes exploration during the early stages. This explains why TRLBO may initially outperform MTRBO but eventually suffers from the risk of premature convergence to local optima.
\par
To illustrate how the proposed method searches the space and gradually converges to the optimal or near-optimal solution during a particular test run on a few benchmark functions, Figures \ref{image:contour_1} and \ref{image:contour_2} are presented. These figures also distinguish between the points selected during the exploration stage and those chosen during the exploitation stage. From these figures, one limitation of the proposed method becomes apparent: it is biased toward exploration. However, this can be mitigated by increasing the number of subiterations in the exploitation stage. Another limitation is the computational complexity of the method—at each iteration, a Gaussian process is fitted to the observed data, which incurs a cost of $\mathcal{O}(N^3)$.
\par
To visualize, after each iteration, how the trust region radius for the exploration stage is changing, Figure \ref{image:exploration_radius} is provided. This gives us insight about the fact that how our method is exploring the search space, as more and more samples are being observed the radius of the trust region is decreasing.
\subsection{Hyper-parameter sensitivity analysis}
In the experiments discussed in this work $\gamma_{\text{inc}}=2, \gamma_{\text{dec}}= 1/2, \eta_{\text{inc}}= 0.01, \delta =0.1$ is fixed throughout. Sensitivity analysis can be done for them also; the increment and decrement coefficients control the size of the trust region based on the improvement of the gradient. If the improvement threshold $\eta_{\text{inc}}= 0.01$ is high, meaning a high improvement of the posterior mean over the current best is needed. So, one can miss a possible better point; instead, the radius is reduced, maybe focusing on a less optimal point. So, a low improvement threshold is needed. The step size in gradient ascent $\delta =0.1$ is fixed, as this does not affect much in this work, as the step is taken in such a way that it will not go outside the current trust region. One can analyze the sensitivity of the hyperparameters $\gamma_{\text{inc}}=2, \gamma_{\text{dec}}= 1/2$, but in this work, they are treated as fixed constants.
In this analysis, the effect of only 5 hyperparameters is analyzed. Those are, the number of initial observations ($n_0$), sampling budget, i.e, the total number of function evaluations ($N$), number of sub-iterations ($N^{'}$) in the exploitation stage, length-scale ($l$) in the kernel function of the Gaussian process, (though in all the experiments length scale is determined using MLE, but it affects the performance) and the initial trust-region radius ($r_0$) for exploitation. The behavior of the mean objective values over 30 trials for different settings of the hyperparameter for the Ackley (2D) function is observed and provided in Table \ref{table:sensitivity_analysis}. A notable observation from the results is that a smaller initial trust region radius ($r_0=0.1 \to 0.5$) often leads to significantly better performance in terms of the final objective value. This suggests that starting with a finer local search region allows the algorithm to exploit high-potential regions more effectively from the beginning. Another important observation is that the length scale of the GP kernel around $0.6 \to 0.8$ results in better performance. It is also evident that, as the number of iterations and the number of sub-iterations increase, the method performs better, which refers to the fact that the method performs more balanced exploration and exploitation. High $N'$ means that the region is locally exploited well, and high $N$ means the search space is explored well. So, high $N$ and $N'$ are giving better results, but that comes with the cost of the complexity of the model. The proposed method is not sensitive to the number of initial points, because with appropriate other hyperparameter values, the method performs very well for different values of $n_0$. In conclusion, a low initial trust region radius, a moderate length scale of the GP kernel, and a high number of iterations and subiterations (in exploitation) are desirable for better performance of the model.
\subsection{Application to Portfolio optimization problem}
Consider a portfolio of $n$ assets,  the main task is to find optimal weights corresponding to the assets. Assume that $R_{1}, R_{2}, \cdots R_{n}$ are the random variables denoting the return for the individual assets. Taking the expectation of the random variables, individual expected return of the assets are found to be, $\pi = [\pi_{1}, \pi_{2},\cdots, \pi_{n}] $, where $\pi_{i} = \mathbb{E}[R_{i}]$ for all $ i= 1,2 \cdots, n$. If $\omega= [\omega_{1}, \omega_{2}, \cdots, \omega_{n}]$ is the weight vector associated with the assets, the random variable denoting the total portfolio return will be $R = \omega_{1}R_{1}+ \omega_{2}R_{2}+ \cdots +\omega_{n}R_{n}$. So, the total expected return of the portfolio will be  
\begin{equation}\label{eq:portfolio_return}
\mathcal{R}=  \sum_{i=1}^{n} \omega_{i}\pi_{i}
\end{equation}
Covariance between the returns of assets $i$ and $j$ is denoted as $\sigma_{ij}$, and defined by $\sigma_{ij}= \rho_{ij}\sigma_{i}\sigma_{j}$ where $\rho_{ij}$ is the correlation coefficient between $R_{i} $ and $R_{j}$ and $\sigma_{i}, \sigma_{j}$ are variance of $R_{i}$ and $R_{j}$ respectively. Therefore, the variance of the total return of the portfolio will be
\begin{equation}\label{eq:Portfolio_risk}
\sigma= \sum_{i=1}^{n}\sum_{j=1}^{n} \omega_{i}\omega_{j}\sigma_{ij}
\end{equation}
Here, Equation (\ref{eq:portfolio_return}) represents the total expected return of the mean-variance portfolio allocation problem, which is desired to be maximized. Whereas Equation (\ref{eq:Portfolio_risk}) refers to the total associated variance taken as a risk measure, which is to be minimized. Subject to the constraints, the weight proportions associated with the assets sum to $1$ and all weights are either positive or zero. A mean-variance portfolio optimization problem is defined as follows 
\begin{equation}\label{eq:Portfolio_objective}
  \hspace{0.8cm}Maximize : f(\omega)= \lambda \mathcal{R}-(1-\lambda)\sigma= \lambda \sum_{i=1}^{n} \omega_{i}\pi_{i} - (1 - \lambda) \sum_{i=1}^{n}\sum_{j=1}^{n} \omega_{i}\omega_{j}\sigma_{ij}  
\end{equation}
$$Subject\hspace{1mm} to,\hspace{3mm} \sum_{i=1}^{n} \omega_{i} =1 , \hspace{3mm}\omega_{i} \geq 0 \hspace{3mm}\forall{i}= 1,2, \cdots, n$$ 
As the number of assets increases, the search spaces increase meaning the problem of optimizing the objective is becoming tougher. The proposed optimization method is then used to solve this problem for high-dimension cases. Two datasets are considered, one with 100 stocks from the Indian National Stock Exchange and another with 200 stocks from the New York Stock Exchange (data from 1st January 2019 to 1st January 2024). The proposed method together with all the method discussed in this work for comparison is used to solve the portfolio optimization problem for both the datasets and the results are provided in Table (\ref{table:portfolio}). The proposed method gives best performance for both datasets based on the mean Sharpe ratio (performance measure) over 30 trials.
\section{Conclusion}\label{Conclusion}
This work proposes a trust region-based BO technique that exploits the Gaussian posterior mean function near the point giving the current best objective function value as well as explores in the region with the highest uncertainty. By exploiting the posterior mean function at each iteration with a few sub-iterations, it avoids calculating the objective function for the exploitation stage while getting a good understanding of that region (as in the long run the mean function approximates the expensive objective function). The balance between exploitation and exploration provides a better idea about the objective function locally as well as covering the search space efficiently. Then the proposed method is compared with state-of-the-art trust region-based Bayesian optimization as well as general Bayesian optimization on a variety of synthetic test functions which are non-convex, non-differentiable, etc. with different dimensions varying from 2 to 500. The solutions of the proposed method are better within a predefined sample budget showing that the proposed approach can generalize the unknown objective function with less number of function evaluations making it suitable for the expensive functions. To show its acceptability to real-world problems, the proposed method is used to solve the portfolio optimization problem with 100 and 200 stocks from different stock markets. The performance of the method is also superior in terms of the Sharpe ratio compared to others.
\par Though this work promises to be a great success, still the method has its limitations; there is scope to generalize this method for multi-objective scenarios where the objectives are conflicting in nature, high computational complexity, and the method is bit explore heavy method but that can be tackled by increasing number of sub iterations in the exploitation stage.

\bibliography{sn-bibliography}

@article{yuan2015recent,
  title={Recent advances in trust region algorithms},
  author={Yuan, Ya-xiang},
  journal={Mathematical Programming},
  volume={151},
  pages={249--281},
  year={2015},
  publisher={Springer}
}

@article{levenberg1944method,
  title={A method for the solution of certain non-linear problems in least squares},
  author={Levenberg, Kenneth},
  journal={Quarterly of applied mathematics},
  volume={2},
  number={2},
  pages={164--168},
  year={1944}
}

@article{marquardt1963algorithm,
  title={An algorithm for least-squares estimation of nonlinear parameters},
  author={Marquardt, Donald W},
  journal={Journal of the society for Industrial and Applied Mathematics},
  volume={11},
  number={2},
  pages={431--441},
  year={1963},
  publisher={SIAM}
}

@article{regis2016trust,
  title={Trust regions in Kriging-based optimization with expected improvement},
  author={Regis, Rommel G},
  journal={Engineering optimization},
  volume={48},
  number={6},
  pages={1037--1059},
  year={2016},
  publisher={Taylor \& Francis}
}

@article{eriksson2019scalable,
  title={Scalable global optimization via local Bayesian optimization},
  author={Eriksson, David and Pearce, Michael and Gardner, Jacob and Turner, Ryan D and Poloczek, Matthias},
  journal={Advances in neural information processing systems},
  volume={32},
  year={2019}
}

@article{li2023trust,
  title={A trust region based local Bayesian optimization without exhausted optimization of acquisition function},
  author={Li, Qingxia and Fu, Anbing and Wei, Wenhong and Zhang, Yuhui},
  journal={Evolving Systems},
  volume={14},
  number={5},
  pages={839--858},
  year={2023},
  publisher={Springer}
}

@article{diouane2023trego,
  title={TREGO: a trust-region framework for efficient global optimization},
  author={Diouane, Youssef and Picheny, Victor and Riche, Rodolophe Le and Perrotolo, Alexandre Scotto Di},
  journal={Journal of Global Optimization},
  volume={86},
  number={1},
  pages={1--23},
  year={2023},
  publisher={Springer}
}

@ARTICLE{10093035,
  author={Lu, Qin and Polyzos, Konstantinos D. and Li, Bingcong and Giannakis, Georgios B.},
  journal={IEEE Transactions on Pattern Analysis and Machine Intelligence}, 
  title={Surrogate Modeling for Bayesian Optimization Beyond a Single Gaussian Process}, 
  year={2023},
  volume={45},
  number={9},
  pages={11283-11296},
  doi={10.1109/TPAMI.2023.3264741}}

@article{kushner1964new,
  title={A new method of locating the maximum point of an arbitrary multipeak curve in the presence of noise},
  author={Kushner, Harold J},
  journal = {J. Basic Eng.},
  volume= {86}, 
  number = {1},
  pages = {97--106},
  year={1964}
}

@inproceedings{movckus1975bayesian,
  title={On Bayesian methods for seeking the extremum},
  author={Mo{\v{c}}kus, Jonas},
  booktitle={Optimization techniques IFIP technical conference: Novosibirsk, July 1--7, 1974},
  pages={400--404},
  year={1975},
  organization={Springer}
}

@article{jones1998efficient,
  title={Efficient global optimization of expensive black-box functions},
  author={Jones, Donald R and Schonlau, Matthias and Welch, William J},
  journal={Journal of Global optimization},
  volume={13},
  pages={455--492},
  year={1998},
  publisher={Springer}
}

@inproceedings{kandasamy2015high,
  title={High dimensional Bayesian optimisation and bandits via additive models},
  author={Kandasamy, Kirthevasan and Schneider, Jeff and P{\'o}czos, Barnab{\'a}s},
  booktitle={International conference on machine learning},
  pages={295--304},
  year={2015},
  organization={PMLR}
}

@inproceedings{gardner2017discovering,
  title={Discovering and exploiting additive structure for Bayesian optimization},
  author={Gardner, Jacob and Guo, Chuan and Weinberger, Kilian and Garnett, Roman and Grosse, Roger},
  booktitle={Artificial Intelligence and Statistics},
  pages={1311--1319},
  year={2017},
  organization={PMLR}
}

@inproceedings{wang2018batched,
  title={Batched large-scale Bayesian optimization in high-dimensional spaces},
  author={Wang, Zi and Gehring, Clement and Kohli, Pushmeet and Jegelka, Stefanie},
  booktitle={International Conference on Artificial Intelligence and Statistics},
  pages={745--754},
  year={2018},
  organization={PMLR}
}

@article{wang2016bayesian,
  title={Bayesian optimization in a billion dimensions via random embeddings},
  author={Wang, Ziyu and Hutter, Frank and Zoghi, Masrour and Matheson, David and De Feitas, Nando},
  journal={Journal of Artificial Intelligence Research},
  volume={55},
  pages={361--387},
  year={2016}
}

@inproceedings{nayebi2019framework,
  title={A framework for Bayesian optimization in embedded subspaces},
  author={Nayebi, Amin and Munteanu, Alexander and Poloczek, Matthias},
  booktitle={International Conference on Machine Learning},
  pages={4752--4761},
  year={2019},
  organization={PMLR}
}

@inproceedings{gonzalez2016batch,
  title={Batch Bayesian optimization via local penalization},
  author={Gonz{\'a}lez, Javier and Dai, Zhenwen and Hennig, Philipp and Lawrence, Neil},
  booktitle={Artificial intelligence and statistics},
  pages={648--657},
  year={2016},
  organization={PMLR}
}

@inproceedings{chevalier2013fast,
  title={Fast computation of the multi-points expected improvement with applications in batch selection},
  author={Chevalier, Cl{\'e}ment and Ginsbourger, David},
  booktitle={International conference on learning and intelligent optimization},
  pages={59--69},
  year={2013},
  organization={Springer}
}

@article{shah2015parallel,
  title={Parallel predictive entropy search for batch global optimization of expensive objective functions},
  author={Shah, Amar and Ghahramani, Zoubin},
  journal={Advances in neural information processing systems},
  volume={28},
  year={2015}
}

@misc{OptimizationTestFunctions,
  title   = {OptimizationTestFunctions: Collection of optimization test functions and some useful methods for working with them},
  author  = {Demetry Pascal},
  year    = {2020},
  howpublished = {\url{https://pypi.org/project/OptimizationTestFunctions/}},
  note    = {Version 1.0.1}
}

@article{du2022bayesian,
  title={Bayesian optimization based dynamic ensemble for time series forecasting},
  author={Du, Liang and Gao, Ruobin and Suganthan, Ponnuthurai Nagaratnam and Wang, David ZW},
  journal={Information Sciences},
  volume={591},
  pages={155--175},
  year={2022},
  publisher={Elsevier}
}

@article{wang2023recent,
  title={Recent advances in Bayesian optimization},
  author={Wang, Xilu and Jin, Yaochu and Schmitt, Sebastian and Olhofer, Markus},
  journal={ACM Computing Surveys},
  volume={55},
  number={13s},
  pages={1--36},
  year={2023},
  publisher={ACM New York, NY}
}

@article{wang2024pre,
  title={Pre-trained Gaussian processes for Bayesian optimization},
  author={Wang, Zi and Dahl, George E and Swersky, Kevin and Lee, Chansoo and Nado, Zachary and Gilmer, Justin and Snoek, Jasper and Ghahramani, Zoubin},
  journal={Journal of Machine Learning Research},
  volume={25},
  number={212},
  pages={1--83},
  year={2024}
}

@book{berlinet2004reproducing,
  title={Reproducing Kernel Hilbert Spaces in Probability and Statistics},
  author={Berlinet, Alain and Thomas-Agnan, Christine},
  year={2004},
  publisher={Springer}
}

@inproceedings{chowdhury2017kernelized,
  title={On kernelized multi-armed bandits},
  author={Chowdhury, Sayak Ray and Gopalan, Aditya},
  booktitle={International Conference on Machine Learning},
  pages={844--853},
  year={2017},
  organization={PMLR}
}

@article{wan2021think,
  title={Think global and act local: Bayesian optimisation over high-dimensional categorical and mixed search spaces},
  author={Wan, Xingchen and Nguyen, Vu and Ha, Huong and Ru, Binxin and Lu, Cong and Osborne, Michael A},
  journal={arXiv preprint arXiv:2102.07188},
  year={2021}
}

@article{bertsekas1997nonlinear,
  title={Nonlinear programming},
  journal={Journal of the Operational Research Society},
  volume={48},
  number={3},
  pages={334--334},
  year={1997},
  publisher={Taylor \& Francis}
}

@inproceedings{namura2021surrogate,
  title={Surrogate-assisted reference vector adaptation to various pareto front shapes for many-objective Bayesian optimization},
  author={Namura, Nobuo},
  booktitle={2021 IEEE Congress on Evolutionary Computation (CEC)},
  pages={901--908},
  year={2021},
  organization={IEEE}
}

@article{chugh2016surrogate,
  title={A surrogate-assisted reference vector guided evolutionary algorithm for computationally expensive many-objective optimization},
  author={Chugh, Tinkle and Jin, Yaochu and Miettinen, Kaisa and Hakanen, Jussi and Sindhya, Karthik},
  journal={IEEE Transactions on Evolutionary Computation},
  volume={22},
  number={1},
  pages={129--142},
  year={2016},
  publisher={IEEE}
}

@article{ament2023unexpected,
  title={Unexpected improvements to expected improvement for bayesian optimization},
  author={Ament, Sebastian and Daulton, Samuel and Eriksson, David and Balandat, Maximilian and Bakshy, Eytan},
  journal={Advances in Neural Information Processing Systems},
  volume={36},
  pages={20577--20612},
  year={2023}
}
\newpage
\begin{appendices}

\section{Test functions}\label{TF}
\begin{table}[htbp]
\centering
\caption{Mathematical Expressions of Optimization Test Functions}\label{tab:test_functions}
\begin{tabular}{|l|p{12cm}|}
\hline
\textbf{Function} & \textbf{Mathematical Expression} \\
\hline
\textbf{Ackley} &
\[
f(\mathbf{x}) = 20 + e - 20 \exp\left(-0.2 \frac{1}{n} \sum_{i=1}^{n} x_i^2 \right) - \exp\left(\frac{1}{n} \sum_{i=1}^{n} \cos(2\pi x_i)\right) , x \in [-4, 4]^d
\]
\\
\hline
\textbf{AckleyTest} &
\[
f(\mathbf{x}) = \sum_{i=1}^{n-1} \left[ 3 \left( \cos(2x_i) + \sin(2x_{i+1}) \right) + e^{-0.2} \sqrt{x_i^2 + x_{i+1}^2} \right], x \in [-30, 30]^d
\]
\\
\hline
\textbf{Rosenbrock} &
\[
f(\mathbf{x}) = \sum_{i=1}^{n-1} \left[ 100(x_{i+1} - x_i^2)^2 + (x_i - 1)^2 \right], x \in [-2, 2]^d
\]
\\
\hline

\textbf{Griewank} &
\[
f(\mathbf{x}) = 1 + \frac{1}{4000} \sum_{i=1}^{n} x_i^2 - \prod_{i=1}^{n} \cos\left(\frac{x_i}{\sqrt{i+1}}\right), x \in [-600, 600]^d
\]
\\
\hline

\textbf{Penalty2} &
\[
f(\mathbf{x}) = 100 \left[ \sum_{i=1}^{d} \max\left(0, |x_i| - 5\right)^4 \right] + 0.1 [ 10 \sin^2(3\pi x_1) + (x_d - 1)^2 \left(1 + \sin(2\pi x_d^2)\right)
\]
\[
+ \sum_{i=1}^{d-1} (x_i - 1)^2 \left(1 + \sin(3\pi x_{i+1}^2)\right) ] , x \in [-60, 60]^d
\]
\\

\hline

\textbf{Quartic} &
\[
f(\mathbf{x}) = \sum_{i=1}^{n} (i+1) \cdot x_i^4 + \text{random}[0,1] , x \in [-1, 1]^d
\]
\\
\hline
\textbf{Rastrigin} &
\[
f(\mathbf{x}) = 10n + \sum_{i=1}^{n} \left( x_i^2 - 10 \cos(2\pi x_i) \right) , x \in [-5, 5]^d
\]
\\
\hline
\textbf{Schwefel Double} &
\[
f(\mathbf{x}) = \sum_{i=1}^{n} \left( \sum_{j=1}^{i} x_j \right)^2, x \in [-60, 60]^d
\]
\\
\hline
\end{tabular}
\end{table}

\begin{table}[htbp]
\centering
\caption{Mathematical Expressions of Optimization Test Functions}\label{tab:test_functions1}
\begin{tabular}{|l|p{13cm}|}
\hline
\textbf{Function} & \textbf{Mathematical Expression} \\
\hline

\textbf{Schwefel Max} &
\[
f(\mathbf{x}) = \max\left( |x_1|, |x_2|, \dots, |x_n| \right), x \in [-100, 100]^d
\]
\\
\hline

\textbf{Schwefel Sin} &
\[
f(\mathbf{x}) = -\sum_{i=1}^{n} x_i \sin\left( \sqrt{|x_i|} \right), x \in [-500, 500]^d
\]
\\
\hline

\textbf{Stairs} &
\[
f(\mathbf{x}) = \sum_{i=1}^{n} \left( \lfloor x_i + 0.5 \rfloor \right)^2, x \in [-6, 6]^d
\]
\\
\hline

\textbf{Abs} &
\[
f(\mathbf{x}) = \sum_{i=1}^{n} |x_i|, x \in [-10, 10]^d
\]
\\
\hline

\textbf{Scheffer} &
\[
f(\mathbf{x}) = 0.5 + \sum_{i=1}^{n-1} \frac{ \left( \sin(x_i^2 - x_{i+1}^2)^2 - 0.5 \right) }{ \left(1 + 0.001(x_i^2 + x_{i+1}^2)\right)^2 }, x \in [-7, 7]^d
\]
\\
\hline
\textbf{Eggholder} &
\[
f(\mathbf{x}) = -\sum_{i=1}^{n-1} \left[ (x_{i+1} + 47) \sin\left( \sqrt{ |x_{i+1} + \frac{x_i}{2} + 47| } \right) + x_i \sin\left( \sqrt{ |x_i - x_{i+1} - 47| } \right)  \right] , 
x \in [-600, 600]^d
\]
\\
\hline

\textbf{Weierstrass} &
\[
f(\mathbf{x}) = \sum_{i=1}^{n} \left[ \sum_{k=0}^{k_{\max}} a^k \cos\left( b^k \pi (x_i + 0.5) \right) \right] - n \sum_{k=0}^{k_{\max}} a^k \cos\left( b^k \pi \cdot 0.5 \right), x \in [-0.6, 0.6]^d
\]
\\
\hline

\end{tabular}
\end{table}
\end{appendices}

\end{document}